\documentclass{amsart}
\usepackage{amsfonts}
\usepackage{graphicx}
\usepackage{amscd}

\newtheorem{theorem}{Theorem}
\theoremstyle{plain}

\newtheorem{corollary}{Corollary}

\newtheorem{example}{Example}

\newtheorem{proposition}{Proposition}
\newtheorem{remark}{Remark}

\numberwithin{equation}{section}

\input{tcilatex}

\begin{document}
\title[Some Inequalities for Trace Class Operators]{Some Inequalities for
Trace Class Operators Via a Kato's Result}
\author{S.S. Dragomir$^{1,2}$}
\address{$^{1}$Mathematics, School of Engineering \& Science\\
Victoria University, PO Box 14428\\
Melbourne City, MC 8001, Australia.}
\email{sever.dragomir@vu.edu.au}
\urladdr{http://rgmia.org/dragomir}
\address{$^{2}$School of Computational \& Applied Mathematics, University of
the Witwatersrand, Private Bag 3, Johannesburg 2050, South Africa}

\begin{abstract}
By the use of the celebrated Kato's inequality we obtain in this paper some
new inequalities for trace class operators on a complex Hilbert space $H.$
Natural applications for functions defined by power series of normal
operators are given as well.
\end{abstract}

\maketitle

\section{Introduction}

We denote by $\mathcal{B}\left( H\right) $ the Banach algebra of all bounded
linear operators on a complex Hilbert space $\left( H;\left\langle \cdot
,\cdot \right\rangle \right) .$

If $P$ is a positive selfadjoint operator on $H,$ i.e. $\left\langle
Px,x\right\rangle \geq 0$ for any $x\in H,$ then the following inequality is
a generalization of the Schwarz inequality in $H$%
\begin{equation}
\left\vert \left\langle Px,y\right\rangle \right\vert ^{2}\leq \left\langle
Px,x\right\rangle \left\langle Py,y\right\rangle ,  \label{e.1.1}
\end{equation}%
for any $x,y\in H.$

The following inequality is of interest as well, see \cite[p. 221]{H}.

Let $P$ be a positive selfadjoint operator on $H.$ Then 
\begin{equation}
\left\Vert Px\right\Vert ^{2}\leq \left\Vert P\right\Vert \left\langle
Px,x\right\rangle  \label{e.1.2}
\end{equation}%
for any $x\in H.$

The \textit{"square root"} of a positive bounded selfadjoint operator on $H$
can be defined as follows, see for instance \cite[p. 240]{H}: \textit{If the
operator }$A\in B\left( H\right) $\textit{\ is selfadjoint and positive,
then there exists a unique positive selfadjoint operator }$B:=\sqrt{A}\in
B\left( H\right) $\textit{\ such that }$B^{2}=A.$\textit{\ If }$A$\textit{\
is invertible, then so is }$B.$

If $A\in \mathcal{B}\left( H\right) ,$ then the operator $A^{\ast }A$ is
selfadjoint and positive. Define the \textit{"absolute value"} operator by $%
\left\vert A\right\vert :=\sqrt{A^{\ast }A}.$

In 1952, Kato \cite{K} proved the following celebrated generalization of
Schwarz inequality for any bounded linear operator $T$ on $H$:%
\begin{equation}
\left\vert \left\langle Tx,y\right\rangle \right\vert ^{2}\leq \left\langle
\left( T^{\ast }T\right) ^{\alpha }x,x\right\rangle \left\langle \left(
TT^{\ast }\right) ^{1-\alpha }y,y\right\rangle ,  \label{e.1.3}
\end{equation}%
for any $x,y\in H,$ $\alpha \in \left[ 0,1\right] .$ Utilizing the modulus
notation introduced before, we can write (\ref{e.1.3}) as follows%
\begin{equation}
\left\vert \left\langle Tx,y\right\rangle \right\vert ^{2}\leq \left\langle
\left\vert T\right\vert ^{2\alpha }x,x\right\rangle \left\langle \left\vert
T^{\ast }\right\vert ^{2\left( 1-\alpha \right) }y,y\right\rangle
\label{e.1.4}
\end{equation}%
for any $x,y\in H,$ $\alpha \in \left[ 0,1\right] .$

It is useful to observe that, if $T=N,$ a normal operator, i.e., we recall
that $NN^{\ast }=N^{\ast }N,$ then the inequality (\ref{e.1.4}) can be
written as%
\begin{equation}
\left\vert \left\langle Nx,y\right\rangle \right\vert ^{2}\leq \left\langle
\left\vert N\right\vert ^{2\alpha }x,x\right\rangle \left\langle \left\vert
N\right\vert ^{2\left( 1-\alpha \right) }y,y\right\rangle ,  \label{e.1.5}
\end{equation}%
and in particular, for selfadjoint operators $A$ we can state it as%
\begin{equation}
\left\vert \left\langle Ax,y\right\rangle \right\vert \leq \left\Vert
\left\vert A\right\vert ^{\alpha }x\right\Vert \left\Vert \left\vert
A\right\vert ^{1-\alpha }y\right\Vert  \label{e.1.6}
\end{equation}%
for any $x,y\in H,$ $\alpha \in \left[ 0,1\right] .$

If $T=U,$ a unitary operator, i.e., we recall that $UU^{\ast }=U^{\ast
}U=1_{H},$ then the inequality (\ref{e.1.4}) becomes%
\begin{equation*}
\left\vert \left\langle Ux,y\right\rangle \right\vert \leq \left\Vert
x\right\Vert \left\Vert y\right\Vert
\end{equation*}%
for any $x,y\in H,$ which provides a natural generalization for the Schwarz
inequality in $H.$

The symmetric powers in the inequalities above are natural to be considered,
so if we choose in (\ref{e.1.4}), (\ref{e.1.5}) and in (\ref{e.1.6}) $\alpha
=1/2$ then we get for any $x,y\in H$ 
\begin{equation}
\left\vert \left\langle Tx,y\right\rangle \right\vert ^{2}\leq \left\langle
\left\vert T\right\vert x,x\right\rangle \left\langle \left\vert T^{\ast
}\right\vert y,y\right\rangle ,  \label{e.1.7}
\end{equation}%
\begin{equation}
\left\vert \left\langle Nx,y\right\rangle \right\vert ^{2}\leq \left\langle
\left\vert N\right\vert x,x\right\rangle \left\langle \left\vert
N\right\vert y,y\right\rangle ,  \label{e.1.8}
\end{equation}%
and%
\begin{equation}
\left\vert \left\langle Ax,y\right\rangle \right\vert \leq \left\Vert
\left\vert A\right\vert ^{1/2}x\right\Vert \left\Vert \left\vert
A\right\vert ^{1/2}y\right\Vert  \label{e.1.9}
\end{equation}%
respectively.

It is also worthwhile to observe that, if we take the supremum over $y\in
H,\left\Vert y\right\Vert =1$ in (\ref{e.1.4}) then we get 
\begin{equation}
\left\Vert Tx\right\Vert ^{2}\leq \left\Vert T\right\Vert ^{2\left( 1-\alpha
\right) }\left\langle \left\vert T\right\vert ^{2\alpha }x,x\right\rangle
\label{e.1.10}
\end{equation}%
for any $x\in H,$ or in an equivalent form%
\begin{equation}
\left\Vert Tx\right\Vert \leq \left\Vert \left\vert T\right\vert ^{\alpha
}x\right\Vert \left\Vert T\right\Vert ^{1-\alpha }  \label{e.1.11}
\end{equation}%
for any $x\in H.$

If we take $\alpha =1/2$ \ in (\ref{e.1.10}), then we get%
\begin{equation}
\left\Vert Tx\right\Vert ^{2}\leq \left\Vert T\right\Vert \left\langle
\left\vert T\right\vert x,x\right\rangle  \label{e.1.12}
\end{equation}%
for any $x\in H,$ which in the particular case of $T=P,$ a positive
operator, provides the result from (\ref{e.1.2}).

For various interesting generalizations, extension and Kato related results,
see the papers \cite{FLN}-\cite{F3}, \cite{L}-\cite{M} and \cite{U}.

In order to state our results concerning new trace inequalities for
operators in Hilbert spaces we need some preliminary facts as follows.

\section{Trace of Operators}

Let $\left( H,\left\langle \cdot ,\cdot \right\rangle \right) $ be a complex
Hilbert space and $\left\{ e_{i}\right\} _{i\in I}$ an \textit{orthonormal
basis} of $H.$ We say that $A\in \mathcal{B}\left( H\right) $ is a \textit{%
Hilbert-Schmidt operator} if%
\begin{equation}
\sum_{i\in I}\left\Vert Ae_{i}\right\Vert ^{2}<\infty .  \label{e.1.1.a}
\end{equation}%
It is well know that, if $\left\{ e_{i}\right\} _{i\in I}$ and $\left\{
f_{j}\right\} _{j\in J}$ are orthonormal bases for $H$ and $A\in \mathcal{B}%
\left( H\right) $ then%
\begin{equation}
\sum_{i\in I}\left\Vert Ae_{i}\right\Vert ^{2}=\sum_{j\in I}\left\Vert
Af_{j}\right\Vert ^{2}=\sum_{j\in I}\left\Vert A^{\ast }f_{j}\right\Vert ^{2}
\label{e.1.2.a}
\end{equation}%
showing that the definition (\ref{e.1.1.a}) is independent of the
orthonormal basis and $A$ is a Hilbert-Schmidt operator iff $A^{\ast }$ is a
Hilbert-Schmidt operator.

Let $\mathcal{B}_{2}\left( H\right) $ the set of Hilbert-Schmidt operators
in $\mathcal{B}\left( H\right) .$ For $A\in \mathcal{B}_{2}\left( H\right) $
we define%
\begin{equation}
\left\Vert A\right\Vert _{2}:=\left( \sum_{i\in I}\left\Vert
Ae_{i}\right\Vert ^{2}\right) ^{1/2}  \label{e.1.3.1}
\end{equation}%
for $\left\{ e_{i}\right\} _{i\in I}$ an orthonormal basis of $H.$ This
definition does not depend on the choice of the orthonormal basis.

Using the triangle inequality in $l^{2}\left( I\right) ,$one checks that $%
\mathcal{B}_{2}\left( H\right) $ is a \textit{vector space} and that $%
\left\Vert \cdot \right\Vert _{2}$ is a norm on $\mathcal{B}_{2}\left(
H\right) ,$ which is usually called in the literature as the \textit{%
Hilbert-Schmidt norm}.

Denote \textit{the modulus} of an operator $A\in \mathcal{B}\left( H\right) $
by $\left\vert A\right\vert :=\left( A^{\ast }A\right) ^{1/2}.$

Because $\left\Vert \left\vert A\right\vert x\right\Vert =\left\Vert
Ax\right\Vert $ for all $x\in H,$ $A$ is Hilbert-Schmidt iff $\left\vert
A\right\vert $ is Hilbert-Schmidt and $\left\Vert A\right\Vert
_{2}=\left\Vert \left\vert A\right\vert \right\Vert _{2}.$ From (\ref%
{e.1.2.a}) we have that if $A\in \mathcal{B}_{2}\left( H\right) ,$ then $%
A^{\ast }\in \mathcal{B}_{2}\left( H\right) $ and $\left\Vert A\right\Vert
_{2}=\left\Vert A^{\ast }\right\Vert _{2}.$

The following theorem collects some of the most important properties of
Hilbert-Schmidt operators:

\begin{theorem}
\label{t.1.1}We have

(i) $\left( \mathcal{B}_{2}\left( H\right) ,\left\Vert \cdot \right\Vert
_{2}\right) $ is a Hilbert space with inner product 
\begin{equation}
\left\langle A,B\right\rangle _{2}:=\sum_{i\in I}\left\langle
Ae_{i},Be_{i}\right\rangle =\sum_{i\in I}\left\langle B^{\ast
}Ae_{i},e_{i}\right\rangle  \label{e.1.4.1}
\end{equation}%
and the definition does not depend on the choice of the orthonormal basis $%
\left\{ e_{i}\right\} _{i\in I}$;

(ii) We have the inequalities 
\begin{equation}
\left\Vert A\right\Vert \leq \left\Vert A\right\Vert _{2}  \label{e.1.4.a}
\end{equation}
for any $A\in \mathcal{B}_{2}\left( H\right) $ and 
\begin{equation}
\left\Vert AT\right\Vert _{2},\left\Vert TA\right\Vert _{2}\leq \left\Vert
T\right\Vert \left\Vert A\right\Vert _{2}  \label{e.1.4.b}
\end{equation}%
for any $A\in \mathcal{B}_{2}\left( H\right) $ and $T\in \mathcal{B}\left(
H\right) ;$

(iii) $\mathcal{B}_{2}\left( H\right) $ is an operator ideal in $\mathcal{B}%
\left( H\right) ,$ i.e. 
\begin{equation*}
\mathcal{B}\left( H\right) \mathcal{B}_{2}\left( H\right) \mathcal{B}\left(
H\right) \subseteq \mathcal{B}_{2}\left( H\right) ;
\end{equation*}

(iv) $\mathcal{B}_{fin}\left( H\right) ,$ the space of operators of finite
rank, is a dense subspace of $\mathcal{B}_{2}\left( H\right) ;$

(v) $\mathcal{B}_{2}\left( H\right) \subseteq \mathcal{K}\left( H\right) ,$
where $\mathcal{K}\left( H\right) $ denotes the algebra of compact operators
on $H.$
\end{theorem}

If $\left\{ e_{i}\right\} _{i\in I}$ an orthonormal basis of $H,$ we say
that $A\in \mathcal{B}\left( H\right) $ is \textit{trace class} if 
\begin{equation}
\left\Vert A\right\Vert _{1}:=\sum_{i\in I}\left\langle \left\vert
A\right\vert e_{i},e_{i}\right\rangle <\infty .  \label{e.1.5.1}
\end{equation}%
The definition of $\left\Vert A\right\Vert _{1}$ does not depend on the
choice of the orthonormal basis $\left\{ e_{i}\right\} _{i\in I}.$ We denote
by $\mathcal{B}_{1}\left( H\right) $ the set of trace class operators in $%
\mathcal{B}\left( H\right) .$

The following proposition holds:

\begin{proposition}
\label{p.1.1}If $A\in \mathcal{B}\left( H\right) ,$ then the following are
equivalent:

(i) $A\in \mathcal{B}_{1}\left( H\right) ;$

(ii) $\left\vert A\right\vert ^{1/2}\in \mathcal{B}_{2}\left( H\right) ;$

(ii) $A$ (or $\left\vert A\right\vert )$ is the product of two elements of $%
\mathcal{B}_{2}\left( H\right) .$
\end{proposition}

The following properties are also well known:

\begin{theorem}
\label{t.1.2}With the above notations:

(i) We have 
\begin{equation}
\left\Vert A\right\Vert _{1}=\left\Vert A^{\ast }\right\Vert _{1}\text{ and }%
\left\Vert A\right\Vert _{2}\leq \left\Vert A\right\Vert _{1}
\label{e.1.6.1}
\end{equation}%
for any $A\in \mathcal{B}_{1}\left( H\right) ;$

(ii) $\mathcal{B}_{1}\left( H\right) $ is an operator ideal in $\mathcal{B}%
\left( H\right) ,$ i.e. 
\begin{equation*}
\mathcal{B}\left( H\right) \mathcal{B}_{1}\left( H\right) \mathcal{B}\left(
H\right) \subseteq \mathcal{B}_{1}\left( H\right) ;
\end{equation*}

(iii) We have%
\begin{equation*}
\mathcal{B}_{2}\left( H\right) \mathcal{B}_{2}\left( H\right) =\mathcal{B}%
_{1}\left( H\right) ;
\end{equation*}

(iv) We have%
\begin{equation*}
\left\Vert A\right\Vert _{1}=\sup \left\{ \left\langle A,B\right\rangle _{2}%
\text{ }|\text{ }B\in \mathcal{B}_{2}\left( H\right) ,\text{ }\left\Vert
B\right\Vert \leq 1\right\} ;
\end{equation*}

(v) $\left( \mathcal{B}_{1}\left( H\right) ,\left\Vert \cdot \right\Vert
_{1}\right) $ is a Banach space.

(iv) We have the following isometric isomorphisms%
\begin{equation*}
\mathcal{B}_{1}\left( H\right) \cong K\left( H\right) ^{\ast }\text{ and }%
\mathcal{B}_{1}\left( H\right) ^{\ast }\cong \mathcal{B}\left( H\right) ,
\end{equation*}%
where $K\left( H\right) ^{\ast }$ is the dual space of $K\left( H\right) $
and $\mathcal{B}_{1}\left( H\right) ^{\ast }$ is the dual space of $\mathcal{%
B}_{1}\left( H\right) .$
\end{theorem}

We define the \textit{trace} of a trace class operator $A\in \mathcal{B}%
_{1}\left( H\right) $ to be%
\begin{equation}
\limfunc{tr}\left( A\right) :=\sum_{i\in I}\left\langle
Ae_{i},e_{i}\right\rangle ,  \label{e.1.7.1}
\end{equation}%
where $\left\{ e_{i}\right\} _{i\in I}$ an orthonormal basis of $H.$ Note
that this coincides with the usual definition of the trace if $H$ is
finite-dimensional. We observe that the series (\ref{e.1.7.1}) converges
absolutely and it is independent from the choice of basis.

The following result collects some properties of the trace:

\begin{theorem}
\label{t.3.1}We have

(i) If $A\in \mathcal{B}_{1}\left( H\right) $ then $A^{\ast }\in \mathcal{B}%
_{1}\left( H\right) $ and 
\begin{equation}
\limfunc{tr}\left( A^{\ast }\right) =\overline{\limfunc{tr}\left( A\right) };
\label{e.1.8.1}
\end{equation}

(ii) If $A\in \mathcal{B}_{1}\left( H\right) $ and $T\in \mathcal{B}\left(
H\right) ,$ then $AT,$ $TA\in \mathcal{B}_{1}\left( H\right) $ and%
\begin{equation}
\limfunc{tr}\left( AT\right) =\limfunc{tr}\left( TA\right) \text{ and }%
\left\vert \limfunc{tr}\left( AT\right) \right\vert \leq \left\Vert
A\right\Vert _{1}\left\Vert T\right\Vert ;  \label{e.1.9.1}
\end{equation}

(iii) $\limfunc{tr}\left( \cdot \right) $ is a bounded linear functional on $%
\mathcal{B}_{1}\left( H\right) $ with $\left\Vert \limfunc{tr}\right\Vert
=1; $

(iv) If $A,$ $B\in \mathcal{B}_{2}\left( H\right) $ then $AB,$ $BA\in 
\mathcal{B}_{1}\left( H\right) $ and $\limfunc{tr}\left( AB\right) =\limfunc{%
tr}\left( BA\right) ;$

(v) $\mathcal{B}_{fin}\left( H\right) $ is a dense subspace of $\mathcal{B}%
_{1}\left( H\right) .$
\end{theorem}

Utilising the trace notation we obviously have that 
\begin{equation*}
\left\langle A,B\right\rangle _{2}=\limfunc{tr}\left( B^{\ast }A\right) =%
\limfunc{tr}\left( AB^{\ast }\right) \text{ and }\left\Vert A\right\Vert
_{2}^{2}=\limfunc{tr}\left( A^{\ast }A\right) =\limfunc{tr}\left( \left\vert
A\right\vert ^{2}\right)
\end{equation*}%
for any $A,$ $B\in \mathcal{B}_{2}\left( H\right) .$

For the theory of trace functionals and their applications the reader is
referred to \cite{Si}.

For some classical trace inequalities see \cite{Ch}, \cite{C}, \cite{N} and 
\cite{Y1}, which are continuations of the work of Bellman \cite{B}. For
related works the reader can refer to \cite{A}, \cite{BJL}, \cite{Ch}, \cite%
{FL}, \cite{Le}, \cite{Li}, \cite{Ma}, \cite{SA0} and \cite{UT}.

\section{Trace Inequalities Via Kato's Result}

We start with the following result:

\begin{theorem}
\label{t.2.1}Let $T\in \mathcal{B}\left( H\right) .$

(i) If for some $\alpha \in \left( 0,1\right) $ we have $\left\vert
T\right\vert ^{2\alpha },$ $\left\vert T^{\ast }\right\vert ^{2\left(
1-\alpha \right) }\in \mathcal{B}_{1}\left( H\right) ,$ then $T\in \mathcal{B%
}_{1}\left( H\right) $ and we have the inequality%
\begin{equation}
\left\vert \limfunc{tr}\left( T\right) \right\vert ^{2}\leq \limfunc{tr}%
\left( \left\vert T\right\vert ^{2\alpha }\right) \limfunc{tr}\left(
\left\vert T^{\ast }\right\vert ^{2\left( 1-\alpha \right) }\right) ;
\label{e.2.1}
\end{equation}

(ii) If for some $\alpha \in \left[ 0,1\right] $ and an orthonormal basis $%
\left\{ e_{i}\right\} _{i\in I}$ the sum 
\begin{equation*}
\sum_{i\in I}\left\Vert Te_{i}\right\Vert ^{\alpha }\left\Vert T^{\ast
}e_{i}\right\Vert ^{1-\alpha }
\end{equation*}
is finite, then $T\in \mathcal{B}_{1}\left( H\right) $ and we have the
inequality%
\begin{equation}
\left\vert \limfunc{tr}\left( T\right) \right\vert \leq \sum_{i\in
I}\left\Vert Te_{i}\right\Vert ^{\alpha }\left\Vert T^{\ast
}e_{i}\right\Vert ^{1-\alpha }.  \label{e.2.2}
\end{equation}

Moreover, if the sums $\sum_{i\in I}\left\Vert Te_{i}\right\Vert $ and $%
\sum_{i\in I}\left\Vert T^{\ast }e_{i}\right\Vert $ are finite for an
orthonormal basis $\left\{ e_{i}\right\} _{i\in I},$ then $T\in \mathcal{B}%
_{1}\left( H\right) $ and we have%
\begin{equation}
\left\vert \limfunc{tr}\left( T\right) \right\vert \leq \inf_{\alpha \in %
\left[ 0,1\right] }\left\{ \sum_{i\in I}\left\Vert Te_{i}\right\Vert
^{\alpha }\left\Vert T^{\ast }e_{i}\right\Vert ^{1-\alpha }\right\} \leq
\min \left\{ \sum_{i\in F}\left\Vert Te_{i}\right\Vert ,\sum_{i\in
F}\left\Vert T^{\ast }e_{i}\right\Vert \right\} .  \label{e.2.2.1}
\end{equation}
\end{theorem}

\begin{proof}
(i) Assume that $\alpha \in \left( 0,1\right) .$ Let $\left\{ e_{i}\right\}
_{i\in I}$ be an orthonormal basis in $H$ and $F$ a finite part of $I.$ Then
by Kato's inequality (\ref{e.1.4}) we have 
\begin{equation}
\left\vert \sum_{i\in F}\left\langle Te_{i},e_{i}\right\rangle \right\vert
\leq \sum_{i\in F}\left\vert \left\langle Te_{i},e_{i}\right\rangle
\right\vert \leq \sum_{i\in F}\left\langle \left\vert T\right\vert ^{2\alpha
}e_{i},e_{i}\right\rangle ^{1/2}\left\langle \left\vert T^{\ast }\right\vert
^{2\left( 1-\alpha \right) }e_{i},e_{i}\right\rangle ^{1/2}.  \label{e.2.3}
\end{equation}%
By Cauchy-Buniakovski-Schwarz inequality for finite sums we have%
\begin{align}
& \sum_{i\in F}\left\langle \left\vert T\right\vert ^{2\alpha
}e_{i},e_{i}\right\rangle ^{1/2}\left\langle \left\vert T^{\ast }\right\vert
^{2\left( 1-\alpha \right) }e_{i},e_{i}\right\rangle ^{1/2}  \label{e.2.4} \\
& \leq \left( \sum_{i\in F}\left[ \left\langle \left\vert T\right\vert
^{2\alpha }e_{i},e_{i}\right\rangle ^{1/2}\right] ^{2}\right) ^{1/2}\left(
\sum_{i\in F}\left[ \left\langle \left\vert T^{\ast }\right\vert ^{2\left(
1-\alpha \right) }e_{i},e_{i}\right\rangle ^{1/2}\right] ^{2}\right) ^{1/2} 
\notag \\
& =\left( \sum_{i\in F}\left\langle \left\vert T\right\vert ^{2\alpha
}e_{i},e_{i}\right\rangle \right) ^{1/2}\left( \sum_{i\in F}\left\langle
\left\vert T^{\ast }\right\vert ^{2\left( 1-\alpha \right)
}e_{i},e_{i}\right\rangle \right) ^{1/2}.  \notag
\end{align}%
Therefore, by (\ref{e.2.3}) and (\ref{e.2.4}) we have 
\begin{equation}
\left\vert \sum_{i\in F}\left\langle Te_{i},e_{i}\right\rangle \right\vert
\leq \left( \sum_{i\in F}\left\langle \left\vert T\right\vert ^{2\alpha
}e_{i},e_{i}\right\rangle \right) ^{1/2}\left( \sum_{i\in F}\left\langle
\left\vert T^{\ast }\right\vert ^{2\left( 1-\alpha \right)
}e_{i},e_{i}\right\rangle \right) ^{1/2}  \label{e.2.5}
\end{equation}%
for any finite part $F$ of $I.$

If for some $\alpha \in \left( 0,1\right) $ we have $\left\vert T\right\vert
^{2\alpha },\left\vert T^{\ast }\right\vert ^{2\left( 1-\alpha \right) }\in 
\mathcal{B}_{1}\left( H\right) ,$ then the sums $\sum_{i\in I}\left\langle
\left\vert T\right\vert ^{2\alpha }e_{i},e_{i}\right\rangle $ and $%
\sum_{i\in I}\left\langle \left\vert T^{\ast }\right\vert ^{2\left( 1-\alpha
\right) }e_{i},e_{i}\right\rangle $ are finite and by (\ref{e.2.5}) we have
that $\sum_{i\in I}\left\langle Te_{i},e_{i}\right\rangle $ is also finite
and we have the inequality (\ref{e.2.1}).

(ii) Assume that $\alpha \in \left[ 0,1\right] .$ Let $\left\{ e_{i}\right\}
_{i\in I}$ be an orthonormal basis in $H$ and $F$ a finite part of $I.$
Utilising McCarthy's inequality for the positive operator $P,$ namely%
\begin{equation*}
\left\langle P^{\beta }x,x\right\rangle \leq \left\langle Px,x\right\rangle
^{\beta },
\end{equation*}%
that holds for $\beta \in \left[ 0,1\right] $ and $x\in H,$ $\left\Vert
x\right\Vert =1,$ we have%
\begin{equation*}
\left\langle \left\vert T\right\vert ^{2\alpha }e_{i},e_{i}\right\rangle
\leq \left\langle \left\vert T\right\vert ^{2}e_{i},e_{i}\right\rangle
^{\alpha }
\end{equation*}%
and%
\begin{equation*}
\left\langle \left\vert T^{\ast }\right\vert ^{2\left( 1-\alpha \right)
}e_{i},e_{i}\right\rangle \leq \left\langle \left\vert T^{\ast }\right\vert
^{2}e_{i},e_{i}\right\rangle ^{1-\alpha }
\end{equation*}%
for any $i\in I.$

Making use of (\ref{e.2.3}) we have 
\begin{align}
\left\vert \sum_{i\in F}\left\langle Te_{i},e_{i}\right\rangle \right\vert &
\leq \sum_{i\in F}\left\vert \left\langle Te_{i},e_{i}\right\rangle
\right\vert \leq \sum_{i\in F}\left\langle \left\vert T\right\vert ^{2\alpha
}e_{i},e_{i}\right\rangle ^{1/2}\left\langle \left\vert T^{\ast }\right\vert
^{2\left( 1-\alpha \right) }e_{i},e_{i}\right\rangle ^{1/2}  \label{e.2.6} \\
& \leq \sum_{i\in F}\left\langle \left\vert T\right\vert
^{2}e_{i},e_{i}\right\rangle ^{\alpha /2}\left\langle \left\vert T^{\ast
}\right\vert ^{2}e_{i},e_{i}\right\rangle ^{\left( 1-\alpha \right) /2} 
\notag \\
& =\sum_{i\in F}\left\langle T^{\ast }Te_{i},e_{i}\right\rangle ^{\alpha
/2}\left\langle TT^{\ast }e_{i},e_{i}\right\rangle ^{\left( 1-\alpha \right)
/2}  \notag \\
& =\sum_{i\in F}\left\Vert Te_{i}\right\Vert ^{\alpha }\left\Vert T^{\ast
}e_{i}\right\Vert ^{1-\alpha }.  \notag
\end{align}%
Utilizing H\"{o}lder's inequality for finite sums and $p=\frac{1}{\alpha },$ 
$q=\frac{1}{1-\alpha }$ we also have 
\begin{align}
& \sum_{i\in F}\left\Vert Te_{i}\right\Vert ^{\alpha }\left\Vert T^{\ast
}e_{i}\right\Vert ^{1-\alpha }  \label{e.2.7} \\
& \leq \left[ \sum_{i\in F}\left( \left\Vert Te_{i}\right\Vert ^{\alpha
}\right) ^{1/\alpha }\right] ^{\alpha }\left[ \sum_{i\in F}\left( \left\Vert
T^{\ast }e_{i}\right\Vert ^{1-\alpha }\right) ^{1/\left( 1-\alpha \right) }%
\right] ^{1-\alpha }  \notag \\
& =\left[ \sum_{i\in F}\left\Vert Te_{i}\right\Vert \right] ^{\alpha }\left[
\sum_{i\in F}\left\Vert T^{\ast }e_{i}\right\Vert \right] ^{1-\alpha }. 
\notag
\end{align}%
Since all the series involved in (\ref{e.2.6}) and (\ref{e.2.7}) are
convergent, then we get%
\begin{align}
\left\vert \sum_{i\in I}\left\langle Te_{i},e_{i}\right\rangle \right\vert &
\leq \sum_{i\in I}\left\Vert Te_{i}\right\Vert ^{\alpha }\left\Vert T^{\ast
}e_{i}\right\Vert ^{1-\alpha }  \label{e.2.7.1} \\
& \leq \left[ \sum_{i\in I}\left\Vert Te_{i}\right\Vert \right] ^{\alpha }%
\left[ \sum_{i\in I}\left\Vert T^{\ast }e_{i}\right\Vert \right] ^{1-\alpha }
\notag
\end{align}%
for any $\alpha \in \left[ 0,1\right] .$

Taking the infimum over $\alpha \in \left[ 0,1\right] $ in (\ref{e.2.7.1})
produces%
\begin{align}
\left\vert \sum_{i\in I}\left\langle Te_{i},e_{i}\right\rangle \right\vert &
\leq \inf_{\alpha \in \left[ 0,1\right] }\left\{ \sum_{i\in F}\left\Vert
Te_{i}\right\Vert ^{\alpha }\left\Vert T^{\ast }e_{i}\right\Vert ^{1-\alpha
}\right\}  \label{e.2.7.2} \\
& \leq \inf_{\alpha \in \left[ 0,1\right] }\left[ \sum_{i\in F}\left\Vert
Te_{i}\right\Vert \right] ^{\alpha }\left[ \sum_{i\in F}\left\Vert T^{\ast
}e_{i}\right\Vert \right] ^{1-\alpha }  \notag \\
& =\min \left\{ \sum_{i\in F}\left\Vert Te_{i}\right\Vert ,\sum_{i\in
F}\left\Vert T^{\ast }e_{i}\right\Vert \right\} .  \notag
\end{align}
\end{proof}

\begin{corollary}
\label{c.2.1}Let $T\in \mathcal{B}\left( H\right) .$

(i) If we have $\left\vert T\right\vert ,$ $\left\vert T^{\ast }\right\vert
\in \mathcal{B}_{1}\left( H\right) ,$ then $T\in \mathcal{B}_{1}\left(
H\right) $ and we have the inequality%
\begin{equation}
\left\vert \limfunc{tr}\left( T\right) \right\vert ^{2}\leq \limfunc{tr}%
\left( \left\vert T\right\vert \right) \limfunc{tr}\left( \left\vert T^{\ast
}\right\vert \right) ;  \label{e.2.8}
\end{equation}

(ii) If for an orthonormal basis $\left\{ e_{i}\right\} _{i\in I}$ the sum $%
\sum_{i\in I}\sqrt{\left\Vert Te_{i}\right\Vert \left\Vert T^{\ast
}e_{i}\right\Vert }$ is finite, then $T\in \mathcal{B}_{1}\left( H\right) $
and we have the inequality%
\begin{equation}
\left\vert \limfunc{tr}\left( T\right) \right\vert \leq \sum_{i\in I}\sqrt{%
\left\Vert Te_{i}\right\Vert \left\Vert T^{\ast }e_{i}\right\Vert }.
\label{e.2.9}
\end{equation}
\end{corollary}

\begin{corollary}
\label{c.2.2}Let $N\in \mathcal{B}\left( H\right) $ be a normal operator. If
for some $\alpha \in \left( 0,1\right) $ we have $\left\vert N\right\vert
^{2\alpha },$ $\left\vert N\right\vert ^{2\left( 1-\alpha \right) }\in 
\mathcal{B}_{1}\left( H\right) ,$ then $N\in \mathcal{B}_{1}\left( H\right) $
and we have the inequality%
\begin{equation}
\left\vert \limfunc{tr}\left( N\right) \right\vert ^{2}\leq \limfunc{tr}%
\left( \left\vert N\right\vert ^{2\alpha }\right) \limfunc{tr}\left(
\left\vert N\right\vert ^{2\left( 1-\alpha \right) }\right) .  \label{e.2.10}
\end{equation}%
In particular, if $\left\vert N\right\vert \in \mathcal{B}_{1}\left(
H\right) ,$ then $N\in \mathcal{B}_{1}\left( H\right) $ and 
\begin{equation}
\left\vert \limfunc{tr}\left( N\right) \right\vert \leq \limfunc{tr}\left(
\left\vert N\right\vert \right) .  \label{e.2.11}
\end{equation}
\end{corollary}

The following result also holds.

\begin{theorem}
\label{t.2.2}Let $T\in \mathcal{B}\left( H\right) $ and $A,B\in \mathcal{B}%
_{2}\left( H\right) .$

(i) For any $\alpha \in \left[ 0,1\right] $ we have $\left\vert A^{\ast
}\right\vert ^{2}\left\vert T\right\vert ^{2\alpha }$, $\left\vert B^{\ast
}\right\vert ^{2}\left\vert T^{\ast }\right\vert ^{2\left( 1-\alpha \right)
} $ and $B^{\ast }TA\in \mathcal{B}_{1}\left( H\right) $ and%
\begin{equation}
\left\vert \limfunc{tr}\left( AB^{\ast }T\right) \right\vert ^{2}\leq 
\limfunc{tr}\left( \left\vert A^{\ast }\right\vert ^{2}\left\vert
T\right\vert ^{2\alpha }\right) \limfunc{tr}\left( \left\vert B^{\ast
}\right\vert ^{2}\left\vert T^{\ast }\right\vert ^{2\left( 1-\alpha \right)
}\right) ;  \label{e.2.12}
\end{equation}

(ii) We also have%
\begin{align}
& \left\vert \limfunc{tr}\left( AB^{\ast }T\right) \right\vert ^{2}
\label{e.2.13} \\
& \leq \min \left\{ \limfunc{tr}\left( \left\vert B\right\vert ^{2}\right) 
\limfunc{tr}\left( \left\vert A^{\ast }\right\vert ^{2}\left\vert
T\right\vert ^{2}\right) ,\limfunc{tr}\left( \left\vert A\right\vert
^{2}\right) \limfunc{tr}\left( \left\vert B^{\ast }\right\vert
^{2}\left\vert T^{\ast }\right\vert ^{2}\right) \right\} .  \notag
\end{align}
\end{theorem}

\begin{proof}
(i) Let $\left\{ e_{i}\right\} _{i\in I}$ be an orthonormal basis in $H$ and 
$F$ a finite part of $I.$ Then by Kato's inequality (\ref{e.1.4}) we have%
\begin{equation}
\left\vert \left\langle TAe_{i},Be_{i}\right\rangle \right\vert ^{2}\leq
\left\langle \left\vert T\right\vert ^{2\alpha }Ae_{i},Ae_{i}\right\rangle
\left\langle \left\vert T^{\ast }\right\vert ^{2\left( 1-\alpha \right)
}Be_{i},Be_{i}\right\rangle  \label{e.2.14}
\end{equation}%
for any $i\in I.$ This is equivalent to%
\begin{equation}
\left\vert \left\langle B^{\ast }TAe_{i},e_{i}\right\rangle \right\vert \leq
\left\langle A^{\ast }\left\vert T\right\vert ^{2\alpha
}Ae_{i},e_{i}\right\rangle ^{1/2}\left\langle B^{\ast }\left\vert T^{\ast
}\right\vert ^{2\left( 1-\alpha \right) }Be_{i},e_{i}\right\rangle ^{1/2}
\label{e.2.14.a}
\end{equation}%
for any $i\in I.$

Using the generalized triangle inequality for the modulus and the
Cauchy-Bunyakowsky-Schwarz inequality for finite sums we have from (\ref%
{e.2.14.a}) that%
\begin{align}
& \left\vert \sum_{i\in F}\left\langle B^{\ast }TAe_{i},e_{i}\right\rangle
\right\vert  \label{e.2.15} \\
& \leq \sum_{i\in F}\left\vert \left\langle B^{\ast
}TAe_{i},e_{i}\right\rangle \right\vert  \notag \\
& \leq \sum_{i\in F}\left\langle A^{\ast }\left\vert T\right\vert ^{2\alpha
}Ae_{i},e_{i}\right\rangle ^{1/2}\left\langle B^{\ast }\left\vert T^{\ast
}\right\vert ^{2\left( 1-\alpha \right) }Be_{i},e_{i}\right\rangle ^{1/2} 
\notag \\
& \leq \left[ \sum_{i\in F}\left( \left\langle A^{\ast }\left\vert
T\right\vert ^{2\alpha }Ae_{i},e_{i}\right\rangle ^{1/2}\right) ^{2}\right]
^{1/2}  \notag \\
& \times \left[ \sum_{i\in F}\left( \left\langle B^{\ast }\left\vert T^{\ast
}\right\vert ^{2\left( 1-\alpha \right) }Be_{i},e_{i}\right\rangle
^{1/2}\right) ^{2}\right] ^{1/2}  \notag \\
& =\left[ \sum_{i\in F}\left\langle A^{\ast }\left\vert T\right\vert
^{2\alpha }Ae_{i},e_{i}\right\rangle \right] ^{1/2}\left[ \sum_{i\in
F}\left\langle B^{\ast }\left\vert T^{\ast }\right\vert ^{2\left( 1-\alpha
\right) }Be_{i},e_{i}\right\rangle \right] ^{1/2}  \notag
\end{align}%
for any $F$ a finite part of $I.$

Let $\alpha \in \left[ 0,1\right] .$ Since $A,B\in \mathcal{B}_{2}\left(
H\right) ,$ then $A^{\ast }\left\vert T\right\vert ^{2\alpha }A$, $B^{\ast
}\left\vert T^{\ast }\right\vert ^{2\left( 1-\alpha \right) }B$ and $B^{\ast
}TA\in \mathcal{B}_{1}\left( H\right) $ and by (\ref{e.2.15}) we have 
\begin{equation}
\left\vert \limfunc{tr}\left( B^{\ast }TA\right) \right\vert \leq \left[ 
\limfunc{tr}\left( A^{\ast }\left\vert T\right\vert ^{2\alpha }A\right) %
\right] ^{1/2}\left[ \limfunc{tr}\left( B^{\ast }\left\vert T^{\ast
}\right\vert ^{2\left( 1-\alpha \right) }B\right) \right] ^{1/2}.
\label{e.2.16}
\end{equation}%
Since, by the properties of trace we have 
\begin{equation*}
\limfunc{tr}\left( B^{\ast }TA\right) =\limfunc{tr}\left( AB^{\ast }T\right)
,
\end{equation*}%
\begin{equation*}
\limfunc{tr}\left( A^{\ast }\left\vert T\right\vert ^{2\alpha }A\right) =%
\limfunc{tr}\left( AA^{\ast }\left\vert T\right\vert ^{2\alpha }\right) =%
\limfunc{tr}\left( \left\vert A^{\ast }\right\vert ^{2}\left\vert
T\right\vert ^{2\alpha }\right)
\end{equation*}%
and%
\begin{equation*}
\limfunc{tr}\left( B^{\ast }\left\vert T^{\ast }\right\vert ^{2\left(
1-\alpha \right) }B\right) =\limfunc{tr}\left( \left\vert B^{\ast
}\right\vert ^{2}\left\vert T^{\ast }\right\vert ^{2\left( 1-\alpha \right)
}\right) ,
\end{equation*}%
then by (\ref{e.2.16}) we get (\ref{e.2.12}).

(ii) Utilising McCarthy's inequality \cite{M} for the positive operator $P$%
\begin{equation*}
\left\langle P^{\beta }x,x\right\rangle \leq \left\langle Px,x\right\rangle
^{\beta }
\end{equation*}%
that holds for $\beta \in \left( 0,1\right) $ and $x\in H,$ $\left\Vert
x\right\Vert =1,$ we have%
\begin{equation}
\left\langle P^{\beta }y,y\right\rangle \leq \left\Vert y\right\Vert
^{2\left( 1-\beta \right) }\left\langle Py,y\right\rangle ^{\beta }
\label{e.2.17}
\end{equation}%
for any $y\in H.$

Let $\left\{ e_{i}\right\} _{i\in I}$ be an orthonormal basis in $H$ and $F$
a finite part of $I.$ From (\ref{e.2.17}) we have%
\begin{equation*}
\left\langle \left\vert T\right\vert ^{2\alpha }Ae_{i},Ae_{i}\right\rangle
\leq \left\Vert Ae_{i}\right\Vert ^{2\left( 1-\alpha \right) }\left\langle
\left\vert T\right\vert ^{2}Ae_{i},Ae_{i}\right\rangle ^{\alpha }
\end{equation*}%
and%
\begin{equation*}
\left\langle \left\vert T^{\ast }\right\vert ^{2\left( 1-\alpha \right)
}Be_{i},Be_{i}\right\rangle \leq \left\Vert Be_{i}\right\Vert ^{2\alpha
}\left\langle \left\vert T^{\ast }\right\vert ^{2}Be_{i},Be_{i}\right\rangle
^{1-\alpha }
\end{equation*}%
for any $i\in I.$

Making use of the inequality (\ref{e.2.14}) we get%
\begin{align*}
\left\vert \left\langle TAe_{i},Be_{i}\right\rangle \right\vert ^{2}& \leq
\left\Vert Ae_{i}\right\Vert ^{2\left( 1-\alpha \right) }\left\langle
\left\vert T\right\vert ^{2}Ae_{i},Ae_{i}\right\rangle ^{\alpha }\left\Vert
Be_{i}\right\Vert ^{2\alpha }\left\langle \left\vert T^{\ast }\right\vert
^{2}Be_{i},Be_{i}\right\rangle ^{1-\alpha } \\
& =\left\Vert Be_{i}\right\Vert ^{2\alpha }\left\langle \left\vert
T\right\vert ^{2}Ae_{i},Ae_{i}\right\rangle ^{\alpha }\left\Vert
Ae_{i}\right\Vert ^{2\left( 1-\alpha \right) }\left\langle \left\vert
T^{\ast }\right\vert ^{2}Be_{i},Be_{i}\right\rangle ^{1-\alpha }
\end{align*}%
and taking the square root we get%
\begin{equation}
\left\vert \left\langle TAe_{i},Be_{i}\right\rangle \right\vert \leq
\left\Vert Be_{i}\right\Vert ^{\alpha }\left\langle \left\vert T\right\vert
^{2}Ae_{i},Ae_{i}\right\rangle ^{\frac{\alpha }{2}}\left\Vert
Ae_{i}\right\Vert ^{1-\alpha }\left\langle \left\vert T^{\ast }\right\vert
^{2}Be_{i},Be_{i}\right\rangle ^{\frac{1-\alpha }{2}}  \label{e.2.18}
\end{equation}%
for any $i\in I.$

Using the generalized triangle inequality for the modulus and the H\"{o}%
lder's inequality for finite sums and $p=\frac{1}{\alpha },$ $q=\frac{1}{%
1-\alpha }$ we get from (\ref{e.2.18}) that%
\begin{align}
& \left\vert \sum_{i\in F}\left\langle B^{\ast }TAe_{i},e_{i}\right\rangle
\right\vert  \label{e.2.19} \\
& \leq \sum_{i\in F}\left\vert \left\langle B^{\ast
}TAe_{i},e_{i}\right\rangle \right\vert  \notag \\
& \leq \sum_{i\in F}\left\Vert Be_{i}\right\Vert ^{\alpha }\left\langle
\left\vert T\right\vert ^{2}Ae_{i},Ae_{i}\right\rangle ^{\frac{\alpha }{2}%
}\left\Vert Ae_{i}\right\Vert ^{1-\alpha }\left\langle \left\vert T^{\ast
}\right\vert ^{2}Be_{i},Be_{i}\right\rangle ^{\frac{1-\alpha }{2}}  \notag \\
& \leq \left( \sum_{i\in F}\left[ \left\Vert Be_{i}\right\Vert ^{\alpha
}\left\langle \left\vert T\right\vert ^{2}Ae_{i},Ae_{i}\right\rangle ^{\frac{%
\alpha }{2}}\right] ^{1/\alpha }\right) ^{\alpha }  \notag \\
& \times \left( \sum_{i\in F}\left[ \left\Vert Ae_{i}\right\Vert ^{1-\alpha
}\left\langle \left\vert T^{\ast }\right\vert ^{2}Be_{i},Be_{i}\right\rangle
^{\frac{1-\alpha }{2}}\right] ^{1/\left( 1-\alpha \right) }\right)
^{1-\alpha }  \notag \\
& =\left( \sum_{i\in F}\left\Vert Be_{i}\right\Vert \left\langle \left\vert
T\right\vert ^{2}Ae_{i},Ae_{i}\right\rangle ^{\frac{1}{2}}\right) ^{\alpha
}\left( \sum_{i\in F}\left\Vert Ae_{i}\right\Vert \left\langle \left\vert
T^{\ast }\right\vert ^{2}Be_{i},Be_{i}\right\rangle ^{\frac{1}{2}}\right)
^{1-\alpha }.  \notag
\end{align}%
By Cauchy-Bunyakowsky-Schwarz inequality for finite sums we also have%
\begin{align*}
\sum_{i\in F}\left\Vert Be_{i}\right\Vert \left\langle \left\vert
T\right\vert ^{2}Ae_{i},Ae_{i}\right\rangle ^{\frac{1}{2}}& \leq \left(
\sum_{i\in F}\left\Vert Be_{i}\right\Vert ^{2}\right) ^{1/2}\left(
\sum_{i\in F}\left\langle \left\vert T\right\vert
^{2}Ae_{i},Ae_{i}\right\rangle \right) ^{1/2} \\
& =\left( \sum_{i\in F}\left\langle \left\vert B\right\vert
^{2}e_{i},e_{i}\right\rangle \right) ^{1/2}\left( \sum_{i\in F}\left\langle
A^{\ast }\left\vert T\right\vert ^{2}Ae_{i},e_{i}\right\rangle \right) ^{1/2}
\end{align*}%
and%
\begin{align*}
\sum_{i\in F}\left\Vert Ae_{i}\right\Vert \left\langle \left\vert T^{\ast
}\right\vert ^{2}Be_{i},Be_{i}\right\rangle ^{\frac{1}{2}}& \leq \left(
\sum_{i\in F}\left\Vert Ae_{i}\right\Vert ^{2}\right) ^{1/2}\left(
\sum_{i\in F}\left\langle \left\vert T^{\ast }\right\vert
^{2}Be_{i},Be_{i}\right\rangle \right) ^{1/2} \\
& =\left( \sum_{i\in F}\left\langle \left\vert A\right\vert
^{2}e_{i},e_{i}\right\rangle \right) ^{1/2}\left( \sum_{i\in F}\left\langle
B^{\ast }\left\vert T^{\ast }\right\vert ^{2}Be_{i},e_{i}\right\rangle
\right) ^{1/2}
\end{align*}%
and by (\ref{e.2.19}) we obtain%
\begin{align}
& \left\vert \sum_{i\in F}\left\langle B^{\ast }TAe_{i},e_{i}\right\rangle
\right\vert  \label{e.2.20} \\
& \leq \left( \sum_{i\in F}\left\langle \left\vert B\right\vert
^{2}e_{i},e_{i}\right\rangle \right) ^{\alpha /2}\left( \sum_{i\in
F}\left\langle A^{\ast }\left\vert T\right\vert
^{2}Ae_{i},e_{i}\right\rangle \right) ^{\alpha /2}  \notag \\
& \times \left( \sum_{i\in F}\left\langle \left\vert A\right\vert
^{2}e_{i},e_{i}\right\rangle \right) ^{\left( 1-\alpha \right) /2}\left(
\sum_{i\in F}\left\langle B^{\ast }\left\vert T^{\ast }\right\vert
^{2}Be_{i},e_{i}\right\rangle \right) ^{\left( 1-\alpha \right) /2}  \notag
\end{align}%
for any $F$ a finite part of $I.$

Let $\alpha \in \left[ 0,1\right] .$ Since $A,B\in \mathcal{B}_{2}\left(
H\right) ,$ then $A^{\ast }\left\vert T\right\vert ^{2}A$ and $B^{\ast
}\left\vert T^{\ast }\right\vert ^{2}B\in \mathcal{B}_{1}\left( H\right) $
and by (\ref{e.2.20}) we get 
\begin{align}
& \left\vert \limfunc{tr}\left( AB^{\ast }T\right) \right\vert ^{2}
\label{e.2.20.1} \\
& \leq \left[ \limfunc{tr}\left( \left\vert B\right\vert ^{2}\right) 
\limfunc{tr}\left( A^{\ast }\left\vert T\right\vert ^{2}A\right) \right]
^{\alpha }\left[ \limfunc{tr}\left( \left\vert A\right\vert ^{2}\right) 
\limfunc{tr}\left( B^{\ast }\left\vert T^{\ast }\right\vert ^{2}B\right) %
\right] ^{1-\alpha }  \notag \\
& =\left[ \limfunc{tr}\left( \left\vert B\right\vert ^{2}\right) \limfunc{tr}%
\left( \left\vert A^{\ast }\right\vert ^{2}\left\vert T\right\vert
^{2}\right) \right] ^{\alpha }\left[ \limfunc{tr}\left( \left\vert
A\right\vert ^{2}\right) \limfunc{tr}\left( \left\vert B^{\ast }\right\vert
^{2}\left\vert T^{\ast }\right\vert ^{2}\right) \right] ^{1-\alpha }.  \notag
\end{align}%
Taking the infimum over $\alpha \in \left[ 0,1\right] $ we get (\ref{e.2.13}%
).
\end{proof}

\begin{corollary}
\label{c.2.3}Let $T\in \mathcal{B}\left( H\right) $ and $A,$ $B\in \mathcal{B%
}_{2}\left( H\right) .$ We have $\left\vert A^{\ast }\right\vert
^{2}\left\vert T\right\vert $, $\left\vert B^{\ast }\right\vert
^{2}\left\vert T^{\ast }\right\vert $ and $B^{\ast }TA\in \mathcal{B}%
_{1}\left( H\right) $ and%
\begin{equation}
\left\vert \limfunc{tr}\left( AB^{\ast }T\right) \right\vert ^{2}\leq 
\limfunc{tr}\left( \left\vert A^{\ast }\right\vert ^{2}\left\vert
T\right\vert \right) \limfunc{tr}\left( \left\vert B^{\ast }\right\vert
^{2}\left\vert T^{\ast }\right\vert \right) .  \label{e.2.21}
\end{equation}
\end{corollary}

\begin{corollary}
\label{c.2.4}Let $N\in \mathcal{B}\left( H\right) $ be a normal operator and 
$A,$ $B\in \mathcal{B}_{2}\left( H\right) .$

(i) For any $\alpha \in \left[ 0,1\right] $ we have $\left\vert A^{\ast
}\right\vert ^{2}\left\vert N\right\vert ^{2\alpha }$, $\left\vert B^{\ast
}\right\vert ^{2}\left\vert N\right\vert ^{2\left( 1-\alpha \right) }$ and $%
B^{\ast }NA\in \mathcal{B}_{1}\left( H\right) $ and%
\begin{equation}
\left\vert \limfunc{tr}\left( AB^{\ast }N\right) \right\vert ^{2}\leq 
\limfunc{tr}\left( \left\vert A^{\ast }\right\vert ^{2}\left\vert
N\right\vert ^{2\alpha }\right) \limfunc{tr}\left( \left\vert B^{\ast
}\right\vert ^{2}\left\vert N\right\vert ^{2\left( 1-\alpha \right) }\right)
.  \label{e.2.23}
\end{equation}%
In particular, we have $\left\vert A^{\ast }\right\vert ^{2}\left\vert
N\right\vert $, $\left\vert B^{\ast }\right\vert ^{2}\left\vert N\right\vert 
$ and $B^{\ast }NA\in \mathcal{B}_{1}\left( H\right) $ and%
\begin{equation}
\left\vert \limfunc{tr}\left( AB^{\ast }N\right) \right\vert ^{2}\leq 
\limfunc{tr}\left( \left\vert A^{\ast }\right\vert ^{2}\left\vert
N\right\vert \right) \limfunc{tr}\left( \left\vert B^{\ast }\right\vert
^{2}\left\vert N\right\vert \right) .  \label{e.2.24}
\end{equation}

(ii) We also have%
\begin{align}
& \left\vert \limfunc{tr}\left( AB^{\ast }N\right) \right\vert ^{2}
\label{e.2.25} \\
& \leq \min \left\{ \limfunc{tr}\left( \left\vert B\right\vert ^{2}\right) 
\limfunc{tr}\left( \left\vert A^{\ast }\right\vert ^{2}\left\vert
N\right\vert ^{2}\right) ,\limfunc{tr}\left( \left\vert A\right\vert
^{2}\right) \limfunc{tr}\left( \left\vert B^{\ast }\right\vert
^{2}\left\vert N\right\vert ^{2}\right) \right\} .  \notag
\end{align}
\end{corollary}

\begin{remark}
\label{r.2.1}Let $\alpha \in \left[ 0,1\right] .$ By replacing $A$ with $%
A^{\ast }$ and $B$ with $B^{\ast }$ in (\ref{e.2.12}) we get 
\begin{equation}
\left\vert \limfunc{tr}\left( A^{\ast }BT\right) \right\vert ^{2}\leq 
\limfunc{tr}\left( \left\vert A\right\vert ^{2}\left\vert T\right\vert
^{2\alpha }\right) \limfunc{tr}\left( \left\vert B\right\vert ^{2}\left\vert
T^{\ast }\right\vert ^{2\left( 1-\alpha \right) }\right)   \label{e.2.26}
\end{equation}%
\textit{for any }$T\in \mathcal{B}\left( H\right) $ and $A,$ $B\in \mathcal{B%
}_{2}\left( H\right) .$

If in this inequality we take $A=B,$ then we get%
\begin{equation}
\left\vert \limfunc{tr}\left( \left\vert B\right\vert ^{2}T\right)
\right\vert ^{2}\leq \limfunc{tr}\left( \left\vert B\right\vert
^{2}\left\vert T\right\vert ^{2\alpha }\right) \limfunc{tr}\left( \left\vert
B\right\vert ^{2}\left\vert T^{\ast }\right\vert ^{2\left( 1-\alpha \right)
}\right)  \label{e.2.27}
\end{equation}%
\textit{for any }$T\in \mathcal{B}\left( H\right) $ and $B\in \mathcal{B}%
_{2}\left( H\right) .$

If in (\ref{e.2.26}) we take $A=B^{\ast },$ then we get%
\begin{equation}
\left\vert \limfunc{tr}\left( B^{2}T\right) \right\vert ^{2}\leq \limfunc{tr}%
\left( \left\vert B^{\ast }\right\vert ^{2}\left\vert T\right\vert ^{2\alpha
}\right) \limfunc{tr}\left( \left\vert B\right\vert ^{2}\left\vert T^{\ast
}\right\vert ^{2\left( 1-\alpha \right) }\right)   \label{e.2.28}
\end{equation}%
\textit{for any }$T\in \mathcal{B}\left( H\right) $ and $B\in \mathcal{B}%
_{2}\left( H\right) .$

Also, if $T=N,$ a normal operator, then (\ref{e.2.27}) and (\ref{e.2.28})
become%
\begin{equation}
\left\vert \limfunc{tr}\left( \left\vert B\right\vert ^{2}N\right)
\right\vert ^{2}\leq \limfunc{tr}\left( \left\vert B\right\vert
^{2}\left\vert N\right\vert ^{2\alpha }\right) \limfunc{tr}\left( \left\vert
B\right\vert ^{2}\left\vert N\right\vert ^{2\left( 1-\alpha \right) }\right)
\label{e.2.29}
\end{equation}%
and%
\begin{equation}
\left\vert \limfunc{tr}\left( B^{2}N\right) \right\vert ^{2}\leq \limfunc{tr}%
\left( \left\vert B^{\ast }\right\vert ^{2}\left\vert N\right\vert ^{2\alpha
}\right) \limfunc{tr}\left( \left\vert B\right\vert ^{2}\left\vert
N\right\vert ^{2\left( 1-\alpha \right) }\right) ,  \label{e.2.30}
\end{equation}%
\textit{for any} $B\in \mathcal{B}_{2}\left( H\right) .$
\end{remark}

\section{Some Functional Properties}

Let $A\in \mathcal{B}_{2}\left( H\right) $ and $P\in \mathcal{B}\left(
H\right) $ with $P\geq 0.$ Then $Q:=A^{\ast }PA\in \mathcal{B}_{1}\left(
H\right) $ with $Q\geq 0$ and writing the inequality (\ref{e.2.27}) for $%
B=\left( A^{\ast }PA\right) ^{1/2}\in \mathcal{B}_{2}\left( H\right) $ we
get 
\begin{equation*}
\left\vert \limfunc{tr}\left( A^{\ast }PAT\right) \right\vert ^{2}\leq 
\limfunc{tr}\left( A^{\ast }PA\left\vert T\right\vert ^{2\alpha }\right) 
\limfunc{tr}\left( A^{\ast }PA\left\vert T^{\ast }\right\vert ^{2\left(
1-\alpha \right) }\right) ,
\end{equation*}%
which, by the properties of trace, is equivalent to%
\begin{equation}
\left\vert \limfunc{tr}\left( PATA^{\ast }\right) \right\vert ^{2}\leq 
\limfunc{tr}\left( PA\left\vert T\right\vert ^{2\alpha }A^{\ast }\right) 
\limfunc{tr}\left( PA\left\vert T^{\ast }\right\vert ^{2\left( 1-\alpha
\right) }A^{\ast }\right) ,  \label{e.3.0.1}
\end{equation}%
where $T\in \mathcal{B}\left( H\right) $ and $\alpha \in \left[ 0,1\right] .$

For a given $A\in \mathcal{B}_{2}\left( H\right) ,$ $T\in \mathcal{B}\left(
H\right) $ and $\alpha \in \left[ 0,1\right] ,$ we consider the functional $%
\sigma _{A,T,\alpha }$ defined on the cone $\mathcal{B}_{+}\left( H\right) $
of nonnegative operators on $\mathcal{B}\left( H\right) $ by%
\begin{align*}
\sigma _{A,T,\alpha }\left( P\right) & :=\left[ \limfunc{tr}\left(
PA\left\vert T\right\vert ^{2\alpha }A^{\ast }\right) \right] ^{1/2}\left[ 
\limfunc{tr}\left( PA\left\vert T^{\ast }\right\vert ^{2\left( 1-\alpha
\right) }A^{\ast }\right) \right] ^{1/2} \\
& -\left\vert \limfunc{tr}\left( PATA^{\ast }\right) \right\vert .
\end{align*}

The following theorem collects some fundamental properties of this
functional.

\begin{theorem}
\label{t.3.2}Let $A\in \mathcal{B}_{2}\left( H\right) ,$ $T\in \mathcal{B}%
\left( H\right) $ and $\alpha \in \left[ 0,1\right] .$

(i) For any $P,$ $Q\in \mathcal{B}_{+}\left( H\right) $ we have%
\begin{equation}
\sigma _{A,T,\alpha }\left( P+Q\right) \geq \sigma _{A,T,\alpha }\left(
P\right) +\sigma _{A,T,\alpha }\left( Q\right) \left( \geq 0\right) ,
\label{e.3.1}
\end{equation}%
namely, $\sigma _{A,T,\alpha }$ is a superadditive functional on $\mathcal{B}%
_{+}\left( H\right) ;$

(ii) For any $P,$ $Q\in \mathcal{B}_{+}\left( H\right) $ with $P\geq Q$ we
have%
\begin{equation}
\sigma _{A,T,\alpha }\left( P\right) \geq \sigma _{A,T,\alpha }\left(
Q\right) \left( \geq 0\right) ,  \label{e.3.2}
\end{equation}%
namely, $\sigma _{A,T,\alpha }$ is a monotonic nondecreasing functional on $%
\mathcal{B}_{+}\left( H\right) ;$

(iii) If $P,$ $Q\in \mathcal{B}_{+}\left( H\right) $ and there exist the
constants $M>m>0$ such that $MQ\geq $ $P\geq mQ$ then%
\begin{equation}
M\sigma _{A,T,\alpha }\left( Q\right) \geq \sigma _{A,T,\alpha }\left(
P\right) \geq m\sigma _{A,T,\alpha }\left( Q\right) \left( \geq 0\right) .
\label{e.3.3}
\end{equation}
\end{theorem}

\begin{proof}
(i) Let $P,$ $Q\in \mathcal{B}_{+}\left( H\right) $. On utilizing the
elementary inequality%
\begin{equation*}
\left( a^{2}+b^{2}\right) ^{1/2}\left( c^{2}+d^{2}\right) ^{1/2}\geq ac+bd,%
\text{ }a,b,c,d\geq 0
\end{equation*}%
and the triangle inequality for the modulus, we have%
\begin{align*}
& \sigma _{A,T,\alpha }\left( P+Q\right)  \\
& =\left[ \limfunc{tr}\left( \left( P+Q\right) A\left\vert T\right\vert
^{2\alpha }A^{\ast }\right) \right] ^{1/2}\left[ \limfunc{tr}\left( \left(
P+Q\right) A\left\vert T^{\ast }\right\vert ^{2\left( 1-\alpha \right)
}A^{\ast }\right) \right] ^{1/2} \\
& -\left\vert \limfunc{tr}\left( \left( P+Q\right) ATA^{\ast }\right)
\right\vert  \\
& =\left[ \limfunc{tr}\left( PA\left\vert T\right\vert ^{2\alpha }A^{\ast
}+QA\left\vert T\right\vert ^{2\alpha }A^{\ast }\right) \right] ^{1/2} \\
& \times \left[ \limfunc{tr}\left( PA\left\vert T^{\ast }\right\vert
^{2\left( 1-\alpha \right) }A^{\ast }+QA\left\vert T^{\ast }\right\vert
^{2\left( 1-\alpha \right) }A^{\ast }\right) \right] ^{1/2} \\
& -\left\vert \limfunc{tr}\left( PATA^{\ast }+QATA^{\ast }\right)
\right\vert  \\
& =\left[ \limfunc{tr}\left( PA\left\vert T\right\vert ^{2\alpha }A^{\ast
}\right) +\limfunc{tr}\left( QA\left\vert T\right\vert ^{2\alpha }A^{\ast
}\right) \right] ^{1/2} \\
& \times \left[ \limfunc{tr}\left( PA\left\vert T^{\ast }\right\vert
^{2\left( 1-\alpha \right) }A^{\ast }\right) +\limfunc{tr}\left(
QA\left\vert T^{\ast }\right\vert ^{2\left( 1-\alpha \right) }A^{\ast
}\right) \right] ^{1/2} \\
& -\left\vert \limfunc{tr}\left( PATA^{\ast }\right) +\limfunc{tr}\left(
QATA^{\ast }\right) \right\vert  \\
& \geq \left[ \limfunc{tr}\left( PA\left\vert T\right\vert ^{2\alpha
}A^{\ast }\right) \right] ^{1/2}\left[ \limfunc{tr}\left( PA\left\vert
T^{\ast }\right\vert ^{2\left( 1-\alpha \right) }A^{\ast }\right) \right]
^{1/2} \\
& +\left[ \limfunc{tr}\left( QA\left\vert T\right\vert ^{2\alpha }A^{\ast
}\right) \right] ^{1/2}\left[ \limfunc{tr}\left( QA\left\vert T^{\ast
}\right\vert ^{2\left( 1-\alpha \right) }A^{\ast }\right) \right] ^{1/2} \\
& -\left\vert \limfunc{tr}\left( PATA^{\ast }\right) \right\vert -\left\vert 
\limfunc{tr}\left( QATA^{\ast }\right) \right\vert  \\
& =\sigma _{A,T,\alpha }\left( P\right) +\sigma _{A,T,\alpha }\left(
Q\right) 
\end{align*}%
and the inequality (\ref{e.3.1}) is proved.

(ii) Let $P,$ $Q\in \mathcal{B}_{+}\left( H\right) $ with $P\geq Q.$
Utilising the superadditivity property we have%
\begin{eqnarray*}
\sigma _{A,T,\alpha }\left( P\right)  &=&\sigma _{A,T,\alpha }\left( \left(
P-Q\right) +Q\right) \geq \sigma _{A,T,\alpha }\left( P-Q\right) +\sigma
_{A,T,\alpha }\left( Q\right)  \\
&\geq &\sigma _{A,T,\alpha }\left( Q\right) 
\end{eqnarray*}%
and the inequality (\ref{e.3.2}) is obtained.

(iii) From the monotonicity property we have 
\begin{equation*}
\sigma _{A,T,\alpha }\left( P\right) \geq \sigma _{A,T,\alpha }\left(
mQ\right) =m\sigma _{A,T,\alpha }\left( Q\right)
\end{equation*}%
and a similar inequality for $M,$ which prove the desired result (\ref{e.3.3}%
).
\end{proof}

\begin{corollary}
\label{c.3.2}Let $A\in \mathcal{B}_{2}\left( H\right) ,$ $T\in \mathcal{B}%
\left( H\right) $ and $\alpha \in \left[ 0,1\right] .$ If $P\in \mathcal{B}%
\left( H\right) $ is such that there exist the constants $M>m>0$ with $%
M1_{H}\geq $ $P\geq m1_{H},$ then we have%
\begin{align}
& M\left( \left[ \limfunc{tr}\left( A\left\vert T\right\vert ^{2\alpha
}A^{\ast }\right) \right] ^{1/2}\left[ \limfunc{tr}\left( A\left\vert
T^{\ast }\right\vert ^{2\left( 1-\alpha \right) }A^{\ast }\right) \right]
^{1/2}-\left\vert \limfunc{tr}\left( ATA^{\ast }\right) \right\vert \right)
\label{e.3.4} \\
& \geq \left[ \limfunc{tr}\left( PA\left\vert T\right\vert ^{2\alpha
}A^{\ast }\right) \right] ^{1/2}\left[ \limfunc{tr}\left( PA\left\vert
T^{\ast }\right\vert ^{2\left( 1-\alpha \right) }A^{\ast }\right) \right]
^{1/2}-\left\vert \limfunc{tr}\left( PATA^{\ast }\right) \right\vert  \notag
\\
& \geq m\left( \left[ \limfunc{tr}\left( A\left\vert T\right\vert ^{2\alpha
}A^{\ast }\right) \right] ^{1/2}\left[ \limfunc{tr}\left( A\left\vert
T^{\ast }\right\vert ^{2\left( 1-\alpha \right) }A^{\ast }\right) \right]
^{1/2}-\left\vert \limfunc{tr}\left( ATA^{\ast }\right) \right\vert \right) .
\notag
\end{align}
\end{corollary}

For a given $A\in \mathcal{B}_{2}\left( H\right) ,$ $T\in \mathcal{B}\left(
H\right) $ and $\alpha \in \left[ 0,1\right] ,$ if we take $P=\left\vert
V\right\vert ^{2}$ with $V\in $ $\mathcal{B}\left( H\right) ,$ we have 
\begin{align*}
\sigma _{A,T,\alpha }\left( \left\vert V\right\vert ^{2}\right) & =\left[ 
\limfunc{tr}\left( \left\vert V\right\vert ^{2}A\left\vert T\right\vert
^{2\alpha }A^{\ast }\right) \right] ^{1/2}\left[ \limfunc{tr}\left(
\left\vert V\right\vert ^{2}A\left\vert T^{\ast }\right\vert ^{2\left(
1-\alpha \right) }A^{\ast }\right) \right] ^{1/2} \\
& -\left\vert \limfunc{tr}\left( \left\vert V\right\vert ^{2}ATA^{\ast
}\right) \right\vert \\
& =\left[ \limfunc{tr}\left( V^{\ast }VA\left\vert T\right\vert ^{2\alpha
}A^{\ast }\right) \right] ^{1/2}\left[ \limfunc{tr}\left( V^{\ast
}VA\left\vert T^{\ast }\right\vert ^{2\left( 1-\alpha \right) }A^{\ast
}\right) \right] ^{1/2} \\
& -\left\vert \limfunc{tr}\left( V^{\ast }VATA^{\ast }\right) \right\vert \\
& =\left[ \limfunc{tr}\left( A^{\ast }V^{\ast }VA\left\vert T\right\vert
^{2\alpha }\right) \right] ^{1/2}\left[ \limfunc{tr}\left( A^{\ast }V^{\ast
}VA\left\vert T^{\ast }\right\vert ^{2\left( 1-\alpha \right) }\right) %
\right] ^{1/2} \\
& -\left\vert \limfunc{tr}\left( A^{\ast }V^{\ast }VAT\right) \right\vert \\
& =\left[ \limfunc{tr}\left( \left( VA\right) ^{\ast }VA\left\vert
T\right\vert ^{2\alpha }\right) \right] ^{1/2}\left[ \limfunc{tr}\left(
\left( VA\right) ^{\ast }VA\left\vert T^{\ast }\right\vert ^{2\left(
1-\alpha \right) }\right) \right] ^{1/2} \\
& -\left\vert \limfunc{tr}\left( \left( VA\right) ^{\ast }VAT\right)
\right\vert \\
& =\left[ \limfunc{tr}\left( \left\vert VA\right\vert ^{2}\left\vert
T\right\vert ^{2\alpha }\right) \right] ^{1/2}\left[ \limfunc{tr}\left(
\left\vert VA\right\vert ^{2}\left\vert T^{\ast }\right\vert ^{2\left(
1-\alpha \right) }\right) \right] ^{1/2}-\left\vert \limfunc{tr}\left(
\left\vert VA\right\vert ^{2}T\right) \right\vert .
\end{align*}

Assume that $A\in \mathcal{B}_{2}\left( H\right) ,$ $T\in \mathcal{B}\left(
H\right) $ and $\alpha \in \left[ 0,1\right] .$

If we use the superadditivity property of the functional $\sigma
_{A,T,\alpha }$ we have for any $V,$ $U\in $ $\mathcal{B}\left( H\right) $
that%
\begin{align}
& \left[ \limfunc{tr}\left( \left( \left\vert VA\right\vert ^{2}+\left\vert
UA\right\vert ^{2}\right) \left\vert T\right\vert ^{2\alpha }\right) \right]
^{1/2}\left[ \limfunc{tr}\left( \left( \left\vert VA\right\vert
^{2}+\left\vert UA\right\vert ^{2}\right) \left\vert T^{\ast }\right\vert
^{2\left( 1-\alpha \right) }\right) \right] ^{1/2}  \label{e.3.5} \\
& -\left\vert \limfunc{tr}\left( \left( \left\vert VA\right\vert
^{2}+\left\vert UA\right\vert ^{2}\right) T\right) \right\vert  \notag \\
& \geq \left[ \limfunc{tr}\left( \left\vert VA\right\vert ^{2}\left\vert
T\right\vert ^{2\alpha }\right) \right] ^{1/2}\left[ \limfunc{tr}\left(
\left\vert VA\right\vert ^{2}\left\vert T^{\ast }\right\vert ^{2\left(
1-\alpha \right) }\right) \right] ^{1/2}-\left\vert \limfunc{tr}\left(
\left\vert VA\right\vert ^{2}T\right) \right\vert  \notag \\
& +\left[ \limfunc{tr}\left( \left\vert UA\right\vert ^{2}\left\vert
T\right\vert ^{2\alpha }\right) \right] ^{1/2}\left[ \limfunc{tr}\left(
\left\vert UA\right\vert ^{2}\left\vert T^{\ast }\right\vert ^{2\left(
1-\alpha \right) }\right) \right] ^{1/2}-\left\vert \limfunc{tr}\left(
\left\vert UA\right\vert ^{2}T\right) \right\vert \left( \geq 0\right) . 
\notag
\end{align}

Also, if $\left\vert V\right\vert ^{2}\geq \left\vert U\right\vert ^{2}$
with $V,$ $U\in \mathcal{B}\left( H\right) ,$ then%
\begin{align}
& \left[ \limfunc{tr}\left( \left\vert VA\right\vert ^{2}\left\vert
T\right\vert ^{2\alpha }\right) \right] ^{1/2}\left[ \limfunc{tr}\left(
\left\vert VA\right\vert ^{2}\left\vert T^{\ast }\right\vert ^{2\left(
1-\alpha \right) }\right) \right] ^{1/2}-\left\vert \limfunc{tr}\left(
\left\vert VA\right\vert ^{2}T\right) \right\vert  \label{e.3.6} \\
& \geq \left[ \limfunc{tr}\left( \left\vert UA\right\vert ^{2}\left\vert
T\right\vert ^{2\alpha }\right) \right] ^{1/2}\left[ \limfunc{tr}\left(
\left\vert UA\right\vert ^{2}\left\vert T^{\ast }\right\vert ^{2\left(
1-\alpha \right) }\right) \right] ^{1/2}-\left\vert \limfunc{tr}\left(
\left\vert UA\right\vert ^{2}T\right) \right\vert \left( \geq 0\right) . 
\notag
\end{align}

If $U\in \mathcal{B}\left( H\right) $ is invertible, then 
\begin{equation*}
\frac{1}{\left\Vert U^{-1}\right\Vert }\left\Vert x\right\Vert \leq
\left\Vert Ux\right\Vert \leq \left\Vert U\right\Vert \left\Vert
x\right\Vert \text{ for any }x\in H,
\end{equation*}%
which implies that%
\begin{equation*}
\frac{1}{\left\Vert U^{-1}\right\Vert ^{2}}1_{H}\leq \left\vert U\right\vert
^{2}\leq \left\Vert U\right\Vert ^{2}1_{H}.
\end{equation*}%
Utilising (\ref{e.3.4}) we get%
\begin{align}
& \left\Vert U\right\Vert ^{2}\left( \left[ \limfunc{tr}\left( \left\vert
A\right\vert ^{2}\left\vert T\right\vert ^{2\alpha }\right) \right] ^{1/2}%
\left[ \limfunc{tr}\left( \left\vert A\right\vert ^{2}\left\vert T^{\ast
}\right\vert ^{2\left( 1-\alpha \right) }\right) \right] ^{1/2}-\left\vert 
\limfunc{tr}\left( \left\vert A\right\vert ^{2}T\right) \right\vert \right) 
\label{e.3.7} \\
& \geq \left[ \limfunc{tr}\left( \left\vert UA\right\vert ^{2}\left\vert
T\right\vert ^{2\alpha }\right) \right] ^{1/2}\left[ \limfunc{tr}\left(
\left\vert UA\right\vert ^{2}\left\vert T^{\ast }\right\vert ^{2\left(
1-\alpha \right) }\right) \right] ^{1/2}-\left\vert \limfunc{tr}\left(
\left\vert UA\right\vert ^{2}T\right) \right\vert   \notag \\
& \geq \frac{1}{\left\Vert U^{-1}\right\Vert ^{2}}\left( \left[ \limfunc{tr}%
\left( \left\vert A\right\vert ^{2}\left\vert T\right\vert ^{2\alpha
}\right) \right] ^{1/2}\left[ \limfunc{tr}\left( \left\vert A\right\vert
^{2}\left\vert T^{\ast }\right\vert ^{2\left( 1-\alpha \right) }\right) %
\right] ^{1/2}-\left\vert \limfunc{tr}\left( \left\vert A\right\vert
^{2}T\right) \right\vert \right) .  \notag
\end{align}

\section{Inequalities for Sequences of Operators}

For $n\geq 2,$ define the Cartesian products $\mathcal{B}^{\left( n\right)
}\left( H\right) :=$ $\mathcal{B}\left( H\right) \times ...\times \mathcal{B}%
\left( H\right) ,$ $\mathcal{B}_{2}^{\left( n\right) }\left( H\right) :=%
\mathcal{B}_{2}\left( H\right) \times ...\times \mathcal{B}_{2}\left(
H\right) $ and $\mathcal{B}_{+}^{\left( n\right) }\left( H\right) :=$ $%
\mathcal{B}_{+}\left( H\right) \times ...\times \mathcal{B}_{+}\left(
H\right) $ where $\mathcal{B}_{+}\left( H\right) $ denotes the convex cone
of nonnegative selfadjoint operators on $H,$ i.e. $P\in \mathcal{B}%
_{+}\left( H\right) $ if $\left\langle Px,x\right\rangle \geq 0$ for any $%
x\in H.$

\begin{proposition}
\label{p.4.1}Let $\mathbf{P}=\left( P_{1},...,P_{n}\right) \in \mathcal{B}%
_{+}^{\left( n\right) }\left( H\right) ,$ $\mathbf{T}=\left(
T_{1},...,T_{n}\right) \in \mathcal{B}^{\left( n\right) }\left( H\right) $, $%
\mathbf{A}=\left( A_{1},...,A_{n}\right) $ $\in \mathcal{B}_{2}^{\left(
n\right) }\left( H\right) $ and $\mathbf{z}=\left( z_{1},...,z_{n}\right)
\in \mathbb{C}^{n}$ with $n\geq 2.$ Then%
\begin{align}
& \left\vert \limfunc{tr}\left(
\sum_{k=1}^{n}z_{k}P_{k}A_{k}T_{k}A_{k}^{\ast }\right) \right\vert ^{2}
\label{e.4.1} \\
& \leq \limfunc{tr}\left( \sum_{k=1}^{n}\left\vert z_{k}\right\vert
P_{k}A_{k}\left\vert T_{k}\right\vert ^{2\alpha }A_{k}^{\ast }\right) 
\limfunc{tr}\left( \sum_{k=1}^{n}\left\vert z_{k}\right\vert
P_{k}A_{k}\left\vert T_{k}^{\ast }\right\vert ^{2\left( 1-\alpha \right)
}A_{k}^{\ast }\right)  \notag
\end{align}%
for any $\alpha \in \left[ 0,1\right] .$
\end{proposition}

\begin{proof}
Using the properties of modulus and the inequality (\ref{e.3.0.1}) we have%
\begin{align*}
& \left\vert \limfunc{tr}\left(
\sum_{k=1}^{n}z_{k}P_{k}A_{k}T_{k}A_{k}^{\ast }\right) \right\vert \\
& =\left\vert \sum_{k=1}^{n}z_{k}\limfunc{tr}\left(
P_{k}A_{k}T_{k}A_{k}^{\ast }\right) \right\vert \leq
\sum_{k=1}^{n}\left\vert z_{k}\right\vert \left\vert \limfunc{tr}\left(
P_{k}A_{k}T_{k}A_{k}^{\ast }\right) \right\vert \\
& \leq \sum_{k=1}^{n}\left\vert z_{k}\right\vert \left[ \limfunc{tr}\left(
P_{k}A_{k}\left\vert T_{k}\right\vert ^{2\alpha }A_{k}^{\ast }\right) \right]
^{1/2}\left[ \limfunc{tr}\left( P_{k}A_{k}\left\vert T_{k}^{\ast
}\right\vert ^{2\left( 1-\alpha \right) }A_{k}^{\ast }\right) \right] ^{1/2}.
\end{align*}%
Utilizing the weighted discrete Cauchy-Bunyakovsky-Schwarz inequality we
also have%
\begin{align*}
& \sum_{k=1}^{n}\left\vert z_{k}\right\vert \left[ \limfunc{tr}\left(
P_{k}A_{k}\left\vert T_{k}\right\vert ^{2\alpha }A_{k}^{\ast }\right) \right]
^{1/2}\left[ \limfunc{tr}\left( P_{k}A_{k}\left\vert T_{k}^{\ast
}\right\vert ^{2\left( 1-\alpha \right) }A_{k}^{\ast }\right) \right] ^{1/2}
\\
& \leq \left( \sum_{k=1}^{n}\left\vert z_{k}\right\vert \left( \left[ 
\limfunc{tr}\left( P_{k}A_{k}\left\vert T_{k}\right\vert ^{2\alpha
}A_{k}^{\ast }\right) \right] ^{1/2}\right) ^{2}\right) ^{1/2} \\
& \times \left( \sum_{k=1}^{n}\left\vert z_{k}\right\vert \left( \left[ 
\limfunc{tr}\left( P_{k}A_{k}\left\vert T_{k}^{\ast }\right\vert ^{2\left(
1-\alpha \right) }A_{k}^{\ast }\right) \right] ^{1/2}\right) ^{2}\right)
^{1/2} \\
& =\left( \sum_{k=1}^{n}\left\vert z_{k}\right\vert \limfunc{tr}\left(
P_{k}A_{k}\left\vert T_{k}\right\vert ^{2\alpha }A_{k}^{\ast }\right)
\right) ^{1/2}\left( \sum_{k=1}^{n}\left\vert z_{k}\right\vert \limfunc{tr}%
\left( P_{k}A_{k}\left\vert T_{k}^{\ast }\right\vert ^{2\left( 1-\alpha
\right) }A_{k}^{\ast }\right) \right) ^{1/2},
\end{align*}%
which imply the desired result (\ref{e.4.1}).
\end{proof}

\begin{remark}
\label{r.4.1}If we take $P_{k}=1_{H}$ for any $k\in \left\{ 1,...,n\right\} $
in (\ref{e.4.1}), then we have the simpler inequality%
\begin{align}
& \left\vert \limfunc{tr}\left( \sum_{k=1}^{n}z_{k}\left\vert
A_{k}\right\vert ^{2}T_{k}\right) \right\vert ^{2}  \label{e.4.1.a} \\
& \leq \limfunc{tr}\left( \sum_{k=1}^{n}\left\vert z_{k}\right\vert
\left\vert A_{k}\right\vert ^{2}\left\vert T_{k}\right\vert ^{2\alpha
}\right) \limfunc{tr}\left( \sum_{k=1}^{n}\left\vert z_{k}\right\vert
\left\vert A_{k}\right\vert ^{2}\left\vert T_{k}^{\ast }\right\vert
^{2\left( 1-\alpha \right) }\right)  \notag
\end{align}%
provided that $\mathbf{T}=\left( T_{1},...,T_{n}\right) \in \mathcal{B}%
^{\left( n\right) }\left( H\right) $, $\mathbf{A}=\left(
A_{1},...,A_{n}\right) $ $\in \mathcal{B}_{2}^{\left( n\right) }\left(
H\right) ,$ $\alpha \in \left[ 0,1\right] $ and $\mathbf{z}=\left(
z_{1},...,z_{n}\right) \in \mathbb{C}^{n}.$
\end{remark}

We consider the functional for $n$-tuples of nonnegative operators $\mathbf{P%
}=\left( P_{1},...,P_{n}\right) \in \mathcal{B}_{+}^{\left( n\right) }\left(
H\right) $ as follows:%
\begin{align}
\sigma _{\mathbf{A},\mathbf{T,}\alpha }\left( \mathbf{P}\right) & :=\left[ 
\limfunc{tr}\left( \sum_{k=1}^{n}P_{k}A_{k}\left\vert T_{k}\right\vert
^{2\alpha }A_{k}^{\ast }\right) \right] ^{1/2}  \label{e.4.2} \\
& \times \left[ \limfunc{tr}\left( \sum_{k=1}^{n}P_{k}A_{k}\left\vert
T_{k}^{\ast }\right\vert ^{2\left( 1-\alpha \right) }A_{k}^{\ast }\right) %
\right] ^{1/2}-\left\vert \limfunc{tr}\left(
\sum_{k=1}^{n}P_{k}A_{k}T_{k}A_{k}^{\ast }\right) \right\vert ,  \notag
\end{align}%
where $\mathbf{T}=\left( T_{1},...,T_{n}\right) \in \mathcal{B}^{\left(
n\right) }\left( H\right) $, $\mathbf{A}=\left( A_{1},...,A_{n}\right) $ $%
\in \mathcal{B}_{2}^{\left( n\right) }\left( H\right) $ and $\alpha \in %
\left[ 0,1\right] .$

Utilising a similar argument to the one in Theorem \ref{t.3.2} we can state:

\begin{proposition}
\label{p.4.2}Let $\mathbf{T}=\left( T_{1},...,T_{n}\right) \in \mathcal{B}%
^{\left( n\right) }\left( H\right) $, $\mathbf{A}=\left(
A_{1},...,A_{n}\right) $ $\in \mathcal{B}_{2}^{\left( n\right) }\left(
H\right) $ and $\alpha \in \left[ 0,1\right] .$

(i) For any $\mathbf{P},$ $\mathbf{Q}\in \mathcal{B}_{+}^{\left( n\right)
}\left( H\right) $ we have%
\begin{equation}
\sigma _{\mathbf{A},\mathbf{T,}\alpha }\left( \mathbf{P}+\mathbf{Q}\right)
\geq \sigma _{\mathbf{A},\mathbf{T,}\alpha }\left( \mathbf{P}\right) +\sigma
_{\mathbf{A},\mathbf{T,}\alpha }\left( \mathbf{Q}\right) \left( \geq
0\right) ,  \label{e.4.3}
\end{equation}%
namely, $\sigma _{\mathbf{A},\mathbf{T,}\alpha }$ is a superadditive
functional on $\mathcal{B}_{+}^{\left( n\right) }\left( H\right) ;$

(ii) For any $\mathbf{P},$ $\mathbf{Q}\in \mathcal{B}_{+}^{\left( n\right)
}\left( H\right) $ with $\mathbf{P}\geq \mathbf{Q,}$ namely $P_{k}\geq Q_{k}$
for all $k\in \left\{ 1,...,n\right\} $ we have%
\begin{equation}
\sigma _{\mathbf{A},\mathbf{T,}\alpha }\left( \mathbf{P}\right) \geq \sigma
_{\mathbf{A},\mathbf{T,}\alpha }\left( \mathbf{Q}\right) \left( \geq
0\right) ,  \label{e.4.4}
\end{equation}%
namely, $\sigma _{\mathbf{A},\mathbf{B}}$ is a monotonic nondecreasing
functional on $\mathcal{B}_{+}^{\left( n\right) }\left( H\right) ;$

(iii) If $\mathbf{P},$ $\mathbf{Q}\in \mathcal{B}_{+}^{\left( n\right)
}\left( H\right) $ and there exist the constants $M>m>0$ such that $M\mathbf{%
Q}\geq $ $\mathbf{P}\geq m\mathbf{Q}$ then%
\begin{equation}
M\sigma _{\mathbf{A},\mathbf{T,}\alpha }\left( \mathbf{Q}\right) \geq \sigma
_{\mathbf{A},\mathbf{T,}\alpha }\left( \mathbf{P}\right) \geq m\sigma _{%
\mathbf{A},\mathbf{T,}\alpha }\left( \mathbf{Q}\right) \left( \geq 0\right) .
\label{e.4.5}
\end{equation}
\end{proposition}

If $\mathbf{P=}\left( p_{1}1_{H},...,p_{n}1_{H}\right) $ with $p_{k}\geq 0,$ 
$k\in \left\{ 1,...,n\right\} $ then the functional of real nonnegative
weights $\mathbf{p=}\left( p_{1},...,p_{n}\right) $ defined by%
\begin{align}
\sigma _{\mathbf{A},\mathbf{T,}\alpha }\left( \mathbf{p}\right) & :=\left[ 
\limfunc{tr}\left( \sum_{k=1}^{n}p_{k}\left\vert A_{k}\right\vert
^{2}\left\vert T_{k}\right\vert ^{2\alpha }\right) \right] ^{1/2}
\label{e.4.6} \\
& \times \left[ \limfunc{tr}\left( \sum_{k=1}^{n}p_{k}\left\vert
A_{k}\right\vert ^{2}\left\vert T_{k}^{\ast }\right\vert ^{2\left( 1-\alpha
\right) }\right) \right] ^{1/2}-\left\vert \limfunc{tr}\left(
\sum_{k=1}^{n}p_{k}\left\vert A_{k}\right\vert ^{2}T_{k}\right) \right\vert 
\notag
\end{align}%
has the same properties as in Theorem \ref{t.3.2}.

Moreover, we have the simple bounds%
\begin{align}
& \max_{k\in \left\{ 1,...,n\right\} }\left\{ p_{k}\right\} \left( \left[ 
\limfunc{tr}\left( \sum_{k=1}^{n}\left\vert A_{k}\right\vert ^{2}\left\vert
T_{k}\right\vert ^{2\alpha }\right) \right] ^{1/2}\right.  \label{e.4.7} \\
& \left. \times \left[ \limfunc{tr}\left( \sum_{k=1}^{n}\left\vert
A_{k}\right\vert ^{2}\left\vert T_{k}^{\ast }\right\vert ^{2\left( 1-\alpha
\right) }\right) \right] ^{1/2}-\left\vert \limfunc{tr}\left(
\sum_{k=1}^{n}p_{k}\left\vert A_{k}\right\vert ^{2}T_{k}\right) \right\vert
\right)  \notag \\
& \geq \left[ \limfunc{tr}\left( \sum_{k=1}^{n}p_{k}\left\vert
A_{k}\right\vert ^{2}\left\vert T_{k}\right\vert ^{2\alpha }\right) \right]
^{1/2}\left[ \limfunc{tr}\left( \sum_{k=1}^{n}p_{k}\left\vert
A_{k}\right\vert ^{2}\left\vert T_{k}^{\ast }\right\vert ^{2\left( 1-\alpha
\right) }\right) \right] ^{1/2}  \notag \\
& -\left\vert \limfunc{tr}\left( \sum_{k=1}^{n}p_{k}\left\vert
A_{k}\right\vert ^{2}T_{k}\right) \right\vert  \notag \\
& \geq \min_{k\in \left\{ 1,...,n\right\} }\left\{ p_{k}\right\} \left( 
\left[ \limfunc{tr}\left( \sum_{k=1}^{n}\left\vert A_{k}\right\vert
^{2}\left\vert T_{k}\right\vert ^{2\alpha }\right) \right] ^{1/2}\right. 
\notag \\
& \left. \times \left[ \limfunc{tr}\left( \sum_{k=1}^{n}\left\vert
A_{k}\right\vert ^{2}\left\vert T_{k}^{\ast }\right\vert ^{2\left( 1-\alpha
\right) }\right) \right] ^{1/2}-\left\vert \limfunc{tr}\left(
\sum_{k=1}^{n}p_{k}\left\vert A_{k}\right\vert ^{2}T_{k}\right) \right\vert
\right) .  \notag
\end{align}

\section{Inequalities for Power Series of Operators}

Denote by:%
\begin{equation*}
D(0,R)=\left\{ 
\begin{array}{ll}
\{z\in \mathbb{C}:\left\vert z\right\vert <R\}, & \quad \text{if $R<\infty $}
\\ 
\mathbb{C}, & \quad \text{if $R=\infty $},%
\end{array}%
\right.
\end{equation*}%
and consider the functions:%
\begin{equation*}
\lambda \mapsto f(\lambda ):D(0,R)\rightarrow \mathbb{C},\text{ }f(\lambda
):=\sum_{n=0}^{\infty }\alpha _{n}\lambda ^{n}
\end{equation*}%
and 
\begin{equation*}
\lambda \mapsto f_{a}(\lambda ):D(0,R)\rightarrow \mathbb{C},\text{ }%
f_{a}(\lambda ):=\sum_{n=0}^{\infty }\left\vert \alpha _{n}\right\vert
\lambda ^{n}.
\end{equation*}

As some natural examples that are useful for applications, we can point out
that, if 
\begin{align}
f\left( \lambda \right) & =\sum_{n=1}^{\infty }\frac{\left( -1\right) ^{n}}{n%
}\lambda ^{n}=\ln \frac{1}{1+\lambda },\text{ }\lambda \in D\left(
0,1\right) ;  \label{E1} \\
g\left( \lambda \right) & =\sum_{n=0}^{\infty }\frac{\left( -1\right) ^{n}}{%
\left( 2n\right) !}\lambda ^{2n}=\cos \lambda ,\text{ }\lambda \in \mathbb{C}%
\text{;}  \notag \\
h\left( \lambda \right) & =\sum_{n=0}^{\infty }\frac{\left( -1\right) ^{n}}{%
\left( 2n+1\right) !}\lambda ^{2n+1}=\sin \lambda ,\text{ }\lambda \in 
\mathbb{C}\text{;}  \notag \\
l\left( \lambda \right) & =\sum_{n=0}^{\infty }\left( -1\right) ^{n}\lambda
^{n}=\frac{1}{1+\lambda },\text{ }\lambda \in D\left( 0,1\right) ;  \notag
\end{align}%
then the corresponding functions constructed by the use of the absolute
values of the coefficients are%
\begin{align}
f_{a}\left( \lambda \right) & =\sum_{n=1}^{\infty }\frac{1}{n}\lambda
^{n}=\ln \frac{1}{1-\lambda },\text{ }\lambda \in D\left( 0,1\right) ;
\label{E2} \\
g_{a}\left( \lambda \right) & =\sum_{n=0}^{\infty }\frac{1}{\left( 2n\right)
!}\lambda ^{2n}=\cosh \lambda ,\text{ }\lambda \in \mathbb{C}\text{;}  \notag
\\
h_{a}\left( \lambda \right) & =\sum_{n=0}^{\infty }\frac{1}{\left(
2n+1\right) !}\lambda ^{2n+1}=\sinh \lambda ,\text{ }\lambda \in \mathbb{C}%
\text{;}  \notag \\
l_{a}\left( \lambda \right) & =\sum_{n=0}^{\infty }\lambda ^{n}=\frac{1}{%
1-\lambda },\text{ }\lambda \in D\left( 0,1\right) .  \notag
\end{align}%
Other important examples of functions as power series representations with
nonnegative coefficients are:%
\begin{align}
\exp \left( \lambda \right) & =\sum_{n=0}^{\infty }\frac{1}{n!}\lambda
^{n}\qquad \lambda \in \mathbb{C}\text{,}  \label{E3} \\
\frac{1}{2}\ln \left( \frac{1+\lambda }{1-\lambda }\right) &
=\sum_{n=1}^{\infty }\frac{1}{2n-1}\lambda ^{2n-1},\qquad \lambda \in
D\left( 0,1\right) ;  \notag \\
\sin ^{-1}\left( \lambda \right) & =\sum_{n=0}^{\infty }\frac{\Gamma \left(
n+\frac{1}{2}\right) }{\sqrt{\pi }\left( 2n+1\right) n!}\lambda
^{2n+1},\qquad \lambda \in D\left( 0,1\right) ;  \notag \\
\tanh ^{-1}\left( \lambda \right) & =\sum_{n=1}^{\infty }\frac{1}{2n-1}%
\lambda ^{2n-1},\qquad \lambda \in D\left( 0,1\right)  \notag \\
_{2}F_{1}\left( \alpha ,\beta ,\gamma ,\lambda \right) & =\sum_{n=0}^{\infty
}\frac{\Gamma \left( n+\alpha \right) \Gamma \left( n+\beta \right) \Gamma
\left( \gamma \right) }{n!\Gamma \left( \alpha \right) \Gamma \left( \beta
\right) \Gamma \left( n+\gamma \right) }\lambda ^{n},\alpha ,\beta ,\gamma
>0,  \notag \\
\lambda & \in D\left( 0,1\right) ;  \notag
\end{align}%
where $\Gamma $ is \textit{Gamma function}.

\begin{theorem}
\label{t.5.1}Let $f(\lambda ):=\sum_{n=1}^{\infty }\alpha _{n}\lambda ^{n}$
be a power series with complex coefficients and convergent on the open disk $%
D\left( 0,R\right) ,$ $R>0.$ Let $N\in \mathcal{B}\left( H\right) $ be a
normal operator. If for some $\alpha \in \left( 0,1\right) $ we have $%
\left\vert N\right\vert ^{2\alpha },$ $\left\vert N\right\vert ^{2\left(
1-\alpha \right) }\in \mathcal{B}_{1}\left( H\right) $ with $\limfunc{tr}%
\left( \left\vert N\right\vert ^{2\alpha }\right) ,$ $\limfunc{tr}\left(
\left\vert N\right\vert ^{2\left( 1-\alpha \right) }\right) <R,$ then we
have the inequality%
\begin{equation}
\left\vert \limfunc{tr}\left( f\left( N\right) \right) \right\vert ^{2}\leq 
\limfunc{tr}\left( f_{a}\left( \left\vert N\right\vert ^{2\alpha }\right)
\right) \limfunc{tr}\left( f_{a}\left( \left\vert N\right\vert ^{2\left(
1-\alpha \right) }\right) \right) .  \label{e.5.1}
\end{equation}
\end{theorem}

\begin{proof}
Since $N$ is a normal operator, then for any natural number $k\geq 1$ we
have $\left\vert N^{k}\right\vert ^{2\alpha }=\left\vert N\right\vert
^{2\alpha k}$ and $\left\vert N^{k}\right\vert ^{2\left( 1-\alpha \right)
}=\left\vert N\right\vert ^{2\left( 1-\alpha \right) k}.$

By the generalized triangle inequality for the modulus we have for $n\geq 2$%
\begin{equation}
\left\vert \limfunc{tr}\left( \sum_{k=1}^{n}\alpha _{k}N^{k}\right)
\right\vert =\left\vert \sum_{k=1}^{n}\alpha _{k}\limfunc{tr}\left(
N^{k}\right) \right\vert \leq \sum_{k=1}^{n}\left\vert \alpha
_{k}\right\vert \left\vert \limfunc{tr}\left( N^{k}\right) \right\vert .
\label{e.5.2}
\end{equation}%
If for some $\alpha \in \left( 0,1\right) $ we have $\left\vert N\right\vert
^{2\alpha },$ $\left\vert N\right\vert ^{2\left( 1-\alpha \right) }\in 
\mathcal{B}_{1}\left( H\right) ,$ then by Corollary \ref{c.2.2} we have $%
N\in \mathcal{B}_{1}\left( H\right) .$ Now, since $N,$ $\left\vert
N\right\vert ^{2\alpha },$ $\left\vert N\right\vert ^{2\left( 1-\alpha
\right) }\in \mathcal{B}_{1}\left( H\right) $ then any natural power of
these operators belong to $\mathcal{B}_{1}\left( H\right) $ and by (\ref%
{e.2.10}) we have 
\begin{equation}
\left\vert \limfunc{tr}\left( N^{k}\right) \right\vert ^{2}\leq \limfunc{tr}%
\left( \left\vert N\right\vert ^{2\alpha k}\right) \limfunc{tr}\left(
\left\vert N\right\vert ^{2\left( 1-\alpha \right) k}\right) ,  \label{e.5.3}
\end{equation}%
for any natural number $k\geq 1.$

Making use of (\ref{e.5.3}) we have%
\begin{equation}
\sum_{k=1}^{n}\left\vert \alpha _{k}\right\vert \left\vert \limfunc{tr}%
\left( N^{k}\right) \right\vert \leq \sum_{k=1}^{n}\left\vert \alpha
_{k}\right\vert \left( \limfunc{tr}\left( \left\vert N\right\vert ^{2\alpha
k}\right) \right) ^{1/2}\left( \limfunc{tr}\left( \left\vert N\right\vert
^{2\left( 1-\alpha \right) k}\right) \right) ^{1/2}.  \label{e.5.4}
\end{equation}%
Utilising the weighted Cauchy-Bunyakovsky-Schwarz inequality for sums we
also have%
\begin{align}
& \sum_{k=1}^{n}\left\vert \alpha _{k}\right\vert \left( \limfunc{tr}\left(
\left\vert N\right\vert ^{2\alpha k}\right) \right) ^{1/2}\left( \limfunc{tr}%
\left( \left\vert N\right\vert ^{2\left( 1-\alpha \right) k}\right) \right)
^{1/2}  \label{e.5.5} \\
& \leq \left[ \sum_{k=1}^{n}\left\vert \alpha _{k}\right\vert \left( \left( 
\limfunc{tr}\left( \left\vert N\right\vert ^{2\alpha k}\right) \right)
^{1/2}\right) ^{2}\right] ^{1/2}  \notag \\
& \times \left[ \sum_{k=1}^{n}\left\vert \alpha _{k}\right\vert \left(
\left( \limfunc{tr}\left( \left\vert N\right\vert ^{2\left( 1-\alpha \right)
k}\right) \right) ^{1/2}\right) ^{2}\right] ^{1/2}  \notag \\
& =\left[ \sum_{k=1}^{n}\left\vert \alpha _{k}\right\vert \limfunc{tr}\left(
\left\vert N\right\vert ^{2\alpha k}\right) \right] ^{1/2}\left[
\sum_{k=1}^{n}\left\vert \alpha _{k}\right\vert \limfunc{tr}\left(
\left\vert N\right\vert ^{2\left( 1-\alpha \right) k}\right) \right] ^{1/2}.
\notag
\end{align}%
Making use of (\ref{e.5.2}), (\ref{e.5.4}) and (\ref{e.5.5}) we get the
inequality%
\begin{equation}
\left\vert \limfunc{tr}\left( \sum_{k=1}^{n}\alpha _{k}N^{k}\right)
\right\vert ^{2}\leq \limfunc{tr}\left( \sum_{k=1}^{n}\left\vert \alpha
_{k}\right\vert \left\vert N\right\vert ^{2\alpha k}\right) \limfunc{tr}%
\left( \sum_{k=1}^{n}\left\vert \alpha _{k}\right\vert \left\vert
N\right\vert ^{2\left( 1-\alpha \right) k}\right)  \label{e.5.6}
\end{equation}%
for any $n\geq 2.$

Due to the fact that $\limfunc{tr}\left( \left\vert N\right\vert ^{2\alpha
}\right) ,$ $\limfunc{tr}\left( \left\vert N\right\vert ^{2\left( 1-\alpha
\right) }\right) <R$ it follows by (\ref{e.2.10}) that $\limfunc{tr}\left(
\left\vert N\right\vert \right) <R$ and the operator series 
\begin{equation*}
\sum_{k=1}^{\infty }\alpha _{k}N^{k},\text{ }\sum_{k=1}^{\infty }\left\vert
\alpha _{k}\right\vert \left\vert N\right\vert ^{2\alpha k}\text{ and }%
\sum_{k=1}^{\infty }\left\vert \alpha _{k}\right\vert \left\vert
N\right\vert ^{2\left( 1-\alpha \right) k}
\end{equation*}%
are convergent in the Banach space $\mathcal{B}_{1}\left( H\right) .$

Taking the limit over $n\rightarrow \infty $ in (\ref{e.5.6}) and using the
continuity of the $\limfunc{tr}\left( \cdot \right) $ on $\mathcal{B}%
_{1}\left( H\right) $ we deduce the desired result (\ref{e.5.1}).
\end{proof}

\begin{example}
\label{Ex.1}a) If we take in $f(\lambda )=\left( 1\pm \lambda \right)
^{-1}-1=\mp \lambda \left( \left( 1\pm \lambda \right) ^{-1}\right) $, $%
\left\vert \lambda \right\vert <1$ then we get from (\ref{e.5.1}) the
inequality%
\begin{align}
& \left\vert \limfunc{tr}\left( N\left( \left( 1_{H}\pm N\right)
^{-1}\right) \right) \right\vert ^{2}  \label{e.5.7} \\
& \leq \limfunc{tr}\left( \left\vert N\right\vert ^{2\alpha }\left(
1_{H}-\left\vert N\right\vert ^{2\alpha }\right) ^{-1}\right) \limfunc{tr}%
\left( \left\vert N\right\vert ^{2\left( 1-\alpha \right) }\left(
1_{H}-\left\vert N\right\vert ^{2\left( 1-\alpha \right) }\right)
^{-1}\right) ,  \notag
\end{align}%
provided that $N\in \mathcal{B}\left( H\right) $ is a normal operator and
for $\alpha \in \left( 0,1\right) $ we have $\left\vert N\right\vert
^{2\alpha },$ $\left\vert N\right\vert ^{2\left( 1-\alpha \right) }\in 
\mathcal{B}_{1}\left( H\right) $ with $\limfunc{tr}\left( \left\vert
N\right\vert ^{2\alpha }\right) ,$ $\limfunc{tr}\left( \left\vert
N\right\vert ^{2\left( 1-\alpha \right) }\right) <1.$

b) If we take in (\ref{e.5.1}) $f(\lambda )=\exp \left( \lambda \right) -1$, 
$\lambda \in \mathbb{C}$ then we get the inequality%
\begin{equation}
\left\vert \limfunc{tr}\left( \exp \left( N\right) -1_{H}\right) \right\vert
^{2}\leq \limfunc{tr}\left( \exp \left( \left\vert N\right\vert ^{2\alpha
}\right) -1_{H}\right) \limfunc{tr}\left( \exp \left( \left\vert
N\right\vert ^{2\left( 1-\alpha \right) }\right) -1_{H}\right) ,
\label{e.5.8}
\end{equation}%
provided that $N\in \mathcal{B}\left( H\right) $ is a normal operator and
for $\alpha \in \left( 0,1\right) $ we have $\left\vert N\right\vert
^{2\alpha },$ $\left\vert N\right\vert ^{2\left( 1-\alpha \right) }\in 
\mathcal{B}_{1}\left( H\right) .$
\end{example}

The following result also holds:

\begin{theorem}
\label{t.5.2}Let $f(\lambda ):=\sum_{n=0}^{\infty }\alpha _{n}\lambda ^{n}$
be a power series with complex coefficients and convergent on the open disk $%
D\left( 0,R\right) ,$ $R>0.$ If $T\in \mathcal{B}\left( H\right) ,$ $A\in 
\mathcal{B}_{2}\left( H\right) $ are normal operators that double commute,
i.e. $TA=AT$ and $TA^{\ast }=A^{\ast }T$ and $\limfunc{tr}\left( \left\vert
A\right\vert ^{2}\left\vert T\right\vert ^{2\alpha }\right) ,$ $\limfunc{tr}%
\left( \left\vert A\right\vert ^{2}\left\vert T\right\vert ^{2\left(
1-\alpha \right) }\right) <R$ for some $\alpha \in \left[ 0,1\right] ,$\ then%
\begin{equation}
\left\vert \limfunc{tr}\left( f\left( \left\vert A\right\vert ^{2}T\right)
\right) \right\vert ^{2}\leq \limfunc{tr}\left( f_{a}\left( \left\vert
A\right\vert ^{2}\left\vert T\right\vert ^{2\alpha }\right) \right) \limfunc{%
tr}\left( f_{a}\left( \left\vert A\right\vert ^{2}\left\vert T\right\vert
^{2\left( 1-\alpha \right) }\right) \right) .  \label{e.5.9}
\end{equation}
\end{theorem}

\begin{proof}
From the inequality (\ref{e.4.1.a}) we have 
\begin{align}
& \left\vert \limfunc{tr}\left( \sum_{k=0}^{n}\alpha _{k}\left\vert
A^{k}\right\vert ^{2}T^{k}\right) \right\vert ^{2}  \label{e.5.10} \\
& \leq \limfunc{tr}\left( \sum_{k=0}^{n}\left\vert \alpha _{k}\right\vert
\left\vert A^{k}\right\vert ^{2}\left\vert T^{k}\right\vert ^{2\alpha
}\right) \limfunc{tr}\left( \sum_{k=0}^{n}\left\vert \alpha _{k}\right\vert
\left\vert A^{k}\right\vert ^{2}\left\vert T^{k}\right\vert ^{2\left(
1-\alpha \right) }\right) .  \notag
\end{align}%
Since $A$ and $T$ are normal operators, then $\left\vert A^{k}\right\vert
^{2}=\left\vert A\right\vert ^{2k},$ $\left\vert T^{k}\right\vert ^{2\alpha
}=\left\vert T\right\vert ^{2\alpha k}$ and $\left\vert T^{k}\right\vert
^{2\left( 1-\alpha \right) }=\left\vert T\right\vert ^{2\left( 1-\alpha
\right) k}$ for any natural number $k\geq 0$ and $\alpha \in \left[ 0,1%
\right] .$

Since $T$ and $A$ double commute, then is easy to see that%
\begin{equation*}
\left\vert A\right\vert ^{2k}T^{k}=\left( \left\vert A\right\vert
^{2}T\right) ^{k},\text{ }\left\vert A\right\vert ^{2k}\left\vert
T\right\vert ^{2\alpha k}=\left( \left\vert A\right\vert ^{2}\left\vert
T\right\vert ^{2\alpha }\right) ^{k}
\end{equation*}%
and%
\begin{equation*}
\left\vert A\right\vert ^{2k}\left\vert T\right\vert ^{2\left( 1-\alpha
\right) k}=\left( \left\vert A\right\vert ^{2}\left\vert T\right\vert
^{2\left( 1-\alpha \right) }\right) ^{k}
\end{equation*}%
for any natural number $k\geq 0$ and $\alpha \in \left[ 0,1\right] .$

Therefore (\ref{e.5.10}) is equivalent to%
\begin{align}
& \left\vert \limfunc{tr}\left( \sum_{k=0}^{n}\alpha _{k}\left( \left\vert
A\right\vert ^{2}T\right) ^{k}\right) \right\vert ^{2}  \label{e.5.11} \\
& \leq \limfunc{tr}\left( \sum_{k=0}^{n}\left\vert \alpha _{k}\right\vert
\left( \left\vert A\right\vert ^{2}\left\vert T\right\vert ^{2\alpha
}\right) ^{k}\right) \limfunc{tr}\left( \sum_{k=0}^{n}\left\vert \alpha
_{k}\right\vert \left( \left\vert A\right\vert ^{2}\left\vert T\right\vert
^{2\left( 1-\alpha \right) }\right) ^{k}\right) ,  \notag
\end{align}%
for any natural number $n\geq 1$and $\alpha \in \left[ 0,1\right] .$

Due to the fact that $\limfunc{tr}\left( \left\vert A\right\vert
^{2}\left\vert T\right\vert ^{2\alpha }\right) ,$ $\limfunc{tr}\left(
\left\vert A\right\vert ^{2}\left\vert T\right\vert ^{2\left( 1-\alpha
\right) }\right) <R$ it follows by (\ref{e.4.1.a}) for $n=1$ that $\limfunc{%
tr}\left( \left\vert A\right\vert ^{2}T\right) <R$ and the operator series 
\begin{equation*}
\sum_{k=1}^{\infty }\alpha _{k}N^{k},\text{ }\sum_{k=1}^{\infty }\left\vert
\alpha _{k}\right\vert \left\vert N\right\vert ^{2\alpha k}\text{ and }%
\sum_{k=1}^{\infty }\left\vert \alpha _{k}\right\vert \left\vert
N\right\vert ^{2\left( 1-\alpha \right) k}
\end{equation*}%
are convergent in the Banach space $\mathcal{B}_{1}\left( H\right) .$

Taking the limit over $n\rightarrow \infty $ in (\ref{e.5.11}) and using the
continuity of the $\limfunc{tr}\left( \cdot \right) $ on $\mathcal{B}%
_{1}\left( H\right) $ we deduce the desired result (\ref{e.5.9}).
\end{proof}

\begin{example}
\label{Ex.2}a) If we take $f(\lambda )=\left( 1\pm \lambda \right) ^{-1}$, $%
\left\vert \lambda \right\vert <1$ then we get from (\ref{e.5.9}) the
inequality%
\begin{align}
& \left\vert \limfunc{tr}\left( \left( 1_{H}\pm \left\vert A\right\vert
^{2}T\right) ^{-1}\right) \right\vert ^{2}  \label{e.5.12} \\
& \leq \limfunc{tr}\left( \left( 1_{H}-\left\vert A\right\vert
^{2}\left\vert T\right\vert ^{2\alpha }\right) ^{-1}\right) \limfunc{tr}%
\left( \left( 1_{H}-\left\vert A\right\vert ^{2}\left\vert T\right\vert
^{2\left( 1-\alpha \right) }\right) ^{-1}\right) ,  \notag
\end{align}%
provided that $T\in \mathcal{B}\left( H\right) ,$ $A\in \mathcal{B}%
_{2}\left( H\right) $ are normal operators that double commute and $\limfunc{%
tr}\left( \left\vert A\right\vert ^{2}\left\vert T\right\vert ^{2\alpha
}\right) ,$ $\limfunc{tr}\left( \left\vert A\right\vert ^{2}\left\vert
T\right\vert ^{2\left( 1-\alpha \right) }\right) <1$ for $\alpha \in \left[
0,1\right] .$

b) If we take in (\ref{e.5.9}) $f(\lambda )=\exp \left( \lambda \right) $, $%
\lambda \in \mathbb{C}$ then we get the inequality%
\begin{equation}
\left\vert \limfunc{tr}\left( \exp \left( \left\vert A\right\vert
^{2}T\right) \right) \right\vert ^{2}\leq \limfunc{tr}\left( \exp \left(
\left\vert A\right\vert ^{2}\left\vert T\right\vert ^{2\alpha }\right)
\right) \limfunc{tr}\left( \exp \left( \left\vert A\right\vert
^{2}\left\vert T\right\vert ^{2\left( 1-\alpha \right) }\right) \right) ,
\label{e.5.13}
\end{equation}%
provided that $T\in \mathcal{B}\left( H\right) $ and $A\in \mathcal{B}%
_{2}\left( H\right) $ are normal operators that double commute and $\alpha
\in \left[ 0,1\right] .$
\end{example}

\begin{theorem}
\label{t.5.3}Let $f\left( z\right) :=\sum_{j=0}^{\infty }p_{j}z^{j}$ and $%
g\left( z\right) :=\sum_{j=0}^{\infty }q_{j}z^{j}$ be two power series with
nonnegative coefficients and convergent on the open disk $D\left( 0,R\right)
,$ $R>0.$ If $T\in \mathcal{B}\left( H\right) ,$ $A\in \mathcal{B}_{2}\left(
H\right) $ are normal operators that double commute and $\limfunc{tr}\left(
\left\vert A\right\vert ^{2}\left\vert T\right\vert ^{2\alpha }\right) ,$ $%
\limfunc{tr}\left( \left\vert A\right\vert ^{2}\left\vert T\right\vert
^{2\left( 1-\alpha \right) }\right) <R$ for $\alpha \in \left[ 0,1\right] ,$%
\ then%
\begin{align}
& \left[ \limfunc{tr}\left( f\left( \left\vert A\right\vert ^{2}\left\vert
T\right\vert ^{2\alpha }\right) +g\left( \left\vert A\right\vert
^{2}\left\vert T\right\vert ^{2\alpha }\right) \right) \right] ^{1/2}
\label{e.5.14} \\
& \times \left[ \limfunc{tr}\left( f\left( \left\vert A\right\vert
^{2}\left\vert T\right\vert ^{2\left( 1-\alpha \right) }\right) +g\left(
\left\vert A\right\vert ^{2}\left\vert T\right\vert ^{2\left( 1-\alpha
\right) }\right) \right) \right] ^{1/2}  \notag \\
& -\left\vert \limfunc{tr}\left( f\left( \left\vert A\right\vert
^{2}T\right) +g\left( \left\vert A\right\vert ^{2}T\right) \right)
\right\vert  \notag \\
& \geq \left[ \limfunc{tr}\left( f\left( \left\vert A\right\vert
^{2}\left\vert T\right\vert ^{2\alpha }\right) \right) \right] ^{1/2}\left[ 
\limfunc{tr}\left( f\left( \left\vert A\right\vert ^{2}\left\vert
T\right\vert ^{2\left( 1-\alpha \right) }\right) \right) \right] ^{1/2} 
\notag \\
& -\left\vert \limfunc{tr}\left( f\left( \left\vert A\right\vert
^{2}T\right) \right) \right\vert  \notag \\
& +\left[ \limfunc{tr}\left( g\left( \left\vert A\right\vert ^{2}\left\vert
T\right\vert ^{2\alpha }\right) \right) \right] ^{1/2}\left[ \limfunc{tr}%
\left( g\left( \left\vert A\right\vert ^{2}\left\vert T\right\vert ^{2\left(
1-\alpha \right) }\right) \right) \right] ^{1/2}  \notag \\
& -\left\vert \limfunc{tr}\left( g\left( \left\vert A\right\vert
^{2}T\right) \right) \right\vert \left( \geq 0\right) .  \notag
\end{align}

Moreover, if $p_{j}\geq q_{j}$ for any $j\in \mathbb{N}$, then, with the
above assumptions on $T$ and $A,$ we have%
\begin{align}
& \left[ \limfunc{tr}\left( f\left( \left\vert A\right\vert ^{2}\left\vert
T\right\vert ^{2\alpha }\right) \right) \right] ^{1/2}\left[ \limfunc{tr}%
\left( f\left( \left\vert A\right\vert ^{2}\left\vert T\right\vert ^{2\left(
1-\alpha \right) }\right) \right) \right] ^{1/2}  \label{e.5.15} \\
& -\left\vert \limfunc{tr}\left( f\left( \left\vert A\right\vert
^{2}T\right) \right) \right\vert  \notag \\
& \geq \left[ \limfunc{tr}\left( g\left( \left\vert A\right\vert
^{2}\left\vert T\right\vert ^{2\alpha }\right) \right) \right] ^{1/2}\left[ 
\limfunc{tr}\left( g\left( \left\vert A\right\vert ^{2}\left\vert
T\right\vert ^{2\left( 1-\alpha \right) }\right) \right) \right] ^{1/2} 
\notag \\
& -\left\vert \limfunc{tr}\left( g\left( \left\vert A\right\vert
^{2}T\right) \right) \right\vert \left( \geq 0\right) .  \notag
\end{align}
\end{theorem}

The proof follows in a similar way to the proof of Theorem \ref{t.5.2} by
making use of the superadditivity and monotonicity properties of the
functional $\sigma _{\mathbf{A},\mathbf{T,}\alpha }\left( \mathbf{\cdot }%
\right) .$ We omit the details.

\begin{example}
\label{Ex.3.2}Now, observe that if we take%
\begin{equation*}
f\left( \lambda \right) =\sinh \lambda =\sum_{n=0}^{\infty }\frac{1}{\left(
2n+1\right) !}\lambda ^{2n+1}
\end{equation*}%
and 
\begin{equation*}
g\left( \lambda \right) =\cosh \lambda =\sum_{n=0}^{\infty }\frac{1}{\left(
2n\right) !}\lambda ^{2n}
\end{equation*}%
then 
\begin{equation*}
f\left( \lambda \right) +g\left( \lambda \right) =\exp \lambda
=\sum_{n=0}^{\infty }\frac{1}{n!}\lambda ^{n}
\end{equation*}
for any $\lambda \in \mathbb{C}$.

If $T\in \mathcal{B}\left( H\right) ,$ $A\in \mathcal{B}_{2}\left( H\right) $
are normal operators that double commute and $\alpha \in \left[ 0,1\right] , 
$ then by (\ref{e.5.14}) we have%
\begin{align}
& \left[ \limfunc{tr}\left( \exp \left( \left\vert A\right\vert
^{2}\left\vert T\right\vert ^{2\alpha }\right) \right) \right] ^{1/2}\left[ 
\limfunc{tr}\left( \exp \left( \left\vert A\right\vert ^{2}\left\vert
T\right\vert ^{2\left( 1-\alpha \right) }\right) \right) \right] ^{1/2}
\label{e.5.16} \\
& -\left\vert \limfunc{tr}\left( \exp \left( \left\vert A\right\vert
^{2}T\right) \right) \right\vert  \notag \\
& \geq \left[ \limfunc{tr}\left( \sinh \left( \left\vert A\right\vert
^{2}\left\vert T\right\vert ^{2\alpha }\right) \right) \right] ^{1/2}\left[ 
\limfunc{tr}\left( \sinh \left( \left\vert A\right\vert ^{2}\left\vert
T\right\vert ^{2\left( 1-\alpha \right) }\right) \right) \right] ^{1/2} 
\notag \\
& -\left\vert \limfunc{tr}\left( \sinh \left( \left\vert A\right\vert
^{2}T\right) \right) \right\vert  \notag \\
& +\left[ \limfunc{tr}\left( \cosh \left( \left\vert A\right\vert
^{2}\left\vert T\right\vert ^{2\alpha }\right) \right) \right] ^{1/2}\left[ 
\limfunc{tr}\left( \cosh \left( \left\vert A\right\vert ^{2}\left\vert
T\right\vert ^{2\left( 1-\alpha \right) }\right) \right) \right] ^{1/2} 
\notag \\
& -\left\vert \limfunc{tr}\left( \cosh \left( \left\vert A\right\vert
^{2}T\right) \right) \right\vert \left( \geq 0\right) .  \notag
\end{align}

Now, consider the series $\frac{1}{1-\lambda }=\sum_{n=0}^{\infty }\lambda
^{n},$ $\lambda \in D\left( 0,1\right) $ and $\ln \frac{1}{1-\lambda }%
=\sum_{n=1}^{\infty }\frac{1}{n}\lambda ^{n},$ $\lambda \in D\left(
0,1\right) $ and define $p_{n}=1,$ $n\geq 0,$ $q_{0}=0,$ $q_{n}=\frac{1}{n},$
$n\geq 1,$ then we observe that for any $n\geq 0$ we have $p_{n}\geq q_{n}.$

If $T\in \mathcal{B}\left( H\right) ,$ $A\in \mathcal{B}_{2}\left( H\right) $
are normal operators that double commute, $\alpha \in \left[ 0,1\right] $
and $\limfunc{tr}\left( \left\vert A\right\vert ^{2}\left\vert T\right\vert
^{2\alpha }\right) ,$ $\limfunc{tr}\left( \left\vert A\right\vert
^{2}\left\vert T\right\vert ^{2\left( 1-\alpha \right) }\right) <1,$ then by
(\ref{e.5.15}) we have%
\begin{align}
& \left[ \limfunc{tr}\left( \left( 1_{H}-\left\vert A\right\vert
^{2}\left\vert T\right\vert ^{2\alpha }\right) ^{-1}\right) \right] ^{1/2}%
\left[ \limfunc{tr}\left( \left( 1_{H}-\left\vert A\right\vert
^{2}\left\vert T\right\vert ^{2\left( 1-\alpha \right) }\right) ^{-1}\right) %
\right] ^{1/2}  \label{e.3.16} \\
& -\left\vert \limfunc{tr}\left( \left( 1_{H}-\left\vert A\right\vert
^{2}T\right) ^{-1}\right) \right\vert   \notag \\
& \geq \left[ \limfunc{tr}\left( \ln \left( 1_{H}-\left\vert A\right\vert
^{2}\left\vert T\right\vert ^{2\alpha }\right) ^{-1}\right) \right] ^{1/2}%
\left[ \limfunc{tr}\left( \ln \left( 1_{H}-\left\vert A\right\vert
^{2}\left\vert T\right\vert ^{2\left( 1-\alpha \right) }\right) ^{-1}\right) %
\right] ^{1/2}  \notag \\
& -\left\vert \limfunc{tr}\left( \ln \left( 1_{H}-\left\vert A\right\vert
^{2}T\right) ^{-1}\right) \right\vert \left( \geq 0\right) .  \notag
\end{align}
\end{example}

\end{document}